\documentclass[11pt, letterpaper]{amsart}
\usepackage[utf8]{inputenc}
\usepackage{mathtools}
\usepackage{amsthm,amsmath,amsfonts,amssymb,amscd,mathrsfs,graphics, caption}
\usepackage{times,float}
\usepackage{amsthm}
\usepackage{mathrsfs}
\usepackage{color}
\usepackage{appendix}
\usepackage{tikz-cd}
\usepackage{tikz}
\usetikzlibrary{arrows}
\usepackage{todonotes}
\usepackage{hyperref}
\usepackage{soul}

\newcommand{\br}[1]{\left\langle#1\right\rangle}  
\newcommand{\brr}[1]{\left\langle\!\left\langle#1\right\rangle\!\right\rangle}

\newcommand{\op}[1]{\operatorname{#1}}

\newtheorem{theorem}{Theorem}[section]
\newtheorem{remark}[theorem]{Remark}

\newtheorem{conjecture}{Conjecture}
\newtheorem{definition}[theorem]{Definition}
\newtheorem{lemma}[theorem]{Lemma}
\newtheorem{prop}[theorem]{Proposition}
\newtheorem{corollary}[theorem]{Corollary}

\newcommand{\wtJ}{\widetilde{J}}

\newcommand{\bt}{\mathbf t}

\newcommand{\AF}{\mathbb{A}}
\newcommand{\ii}{\mathbf{1}}
\newcommand{\sMbar}{\overline{\mathcal M}}
 \newcommand{\tot}{\operatorname{tot}}
 \newcommand{\FM}{\mathbb{FM}}
 \newcommand{\TT}{\mathbb{T}}
  \newcommand{\ST}{\mathbb{S}}
\newcommand{\CC}{\mathbb{C}}
\newcommand{\UU}{\mathbb{U}}
\newcommand{\VV}{\mathbb{V}}
\newcommand{\RR}{\mathbb{R}}
\newcommand{\ZZ}{\mathbb{Z}}
\newcommand{\PP}{\mathbb{P}}
\newcommand{\one}{\mathbf 1}

\newcommand{\cc}{\mathcal }

\newcommand{\rank}{\operatorname{rank}}
\newcommand{\ch}{\operatorname{Ch}}

\usepackage[margin=1.2in]{geometry}

\usepackage{fancyhdr}

\begin{document}

\title{Quantum spectrum and Gamma structure 
for standard flips}
\author{Yefeng Shen, Mark Shoemaker}

\maketitle

\begin{abstract}

We investigate the quantum spectrum and Gamma structure for projective bundles, blow-ups, and standard flips. 
After restricting the quantum multiplication to the extremal curve direction, we obtain a decomposition of the quantum cohomology of standard flips into asymptotic Gamma classes.
We then show that this decomposition is compatible with the semi-orthogonal decompositions for these spaces constructed in work of Orlov and Belmans-Fu-Raedschelders.  The proof involves a sequence of reductions to a local model and the asymptotic behavior of Meijer G-functions.

\end{abstract}

{\setcounter{tocdepth}{1} \tableofcontents}

\section{Introduction}

Let $X$ be a smooth projective variety over $\mathbb{C}$.
For $\alpha, \beta \in H^*(X)=H^*(X, \mathbb{C})$, the  quantum product $\alpha \star_{(Q, \bt)} \beta$ is a deformation of the cup product depending on both a Novikov parameter $Q$ and a 
class $\bt \in H^*(X).$  
Choose a homogeneous basis $\{\phi_i\}$ of $H^*(X)$ and let $\bt = \sum_{i=1}^{n} t^i \phi_i.$ Define the \emph{Euler vector field} to be
$$E(\bt) =c_1(X) + \sum_{i=1}^{n} \left(1 - {\deg(\phi_i)\over 2}\right)t^i \phi_i.$$
The \emph{quantum spectrum} of $X$ (at $(Q, \bt)$) is  the spectrum of quantum multiplication by $E(\bt)$, that is, it is the set of eigenvalues of the operator $$E(\bt) \star_{Q, \bt} - : H^*(X) \to H^*(X)[[Q, \bt]].$$  When $\bt = 0$, $E(\bt) \star_{Q, \bt} -$ is just the operator of (small) quantum multiplication by $c_1(X).$  
 
The \emph{quantum connection} on the trivial $H^*(X)$-bundle over $\PP^1$ is given by
\begin{equation}
\label{quantum-connection}
\nabla_{z {d\over d z}} = z\frac{d}{d z} - {E(\bt) \star_{Q, \bt}- \over z} + \mu,
\end{equation}
where $\mu$ is the \emph{Hodge grading operator} of $X$ (defined in \eqref{hodge-grading}).
It is a meremorphic flat connection with a regular singularity at $z = \infty$ and an irregular singularity at $z = 0$ if $E(\bt) \star_{Q, \bt}-$ is nontrivial.

The conjectural relationship between the quantum cohomology of $X$ and the derived category of coherent sheaves $D^b(X)$ has a long history.
 In his 1998 ICM talk \cite{Dub-icm}, Dubrovin conjectured that a Fano variety $X$ has semisimple quantum cohomology if and only if $D^b(X)$ possesses a full exceptional collection. Via Gamma structures \cite{Hos, KKP08, Iri-integral} and the quantum spectrum, Dubrovin's conjecture has been revisited more recently in Gamma Conjecture II  of Galkin-Golyshev-Iritani's \cite{GGI} and in a refined conjecture by Cotti-Dubrovin-Guzzetti in \cite{CDG}. 
Beyond the semisimple case, Sanda-Shamoto formulated a  general ``Dubrovin-type'' conjecture in \cite{SS}.
The relationship between the quantum spectrum and $D^b(X)$ 
is also closely connected to
Kontsevich's Homological Mirror Symmetry conjecture \cite{Kon-icm}, which relates $D^b(X)$ to the Fukaya-Seidel category of a mirror of $X$, if the latter exists. For example, when $X$ is a smooth toric Fano variety, Auroux showed in \cite{Aur} that the eigenvalues of $c_1(X) \star_Q -$ are critical values of the mirror Landau-Ginzburg potential.

 Assume for simplicity that the quantum product is convergent for sufficiently small values of $Q$ and $\bt.$ 
Consider the  formal scheme $\cc M^{alg} \subset \op{Specf}\mathbb{Q}[[Q, t^1, \ldots, t^n]]$ defined by restricting $\bt$ to the subspace of algebraic classes $H^*_{alg}(X) \subset H^*(X)$.
In his 2021 talk \emph{Blow-up formula for quantum cohomology}  \cite{Kon21}, Kontsevich proposed the following
  ``very optimistic'' conjecture.  
\begin{conjecture}\label{c:K1}
    For any point $(Q, \bt)$ in  $\cc M^{alg}$ and a choice of disjoint paths (called Gabrielov paths) from $- \infty$ to the eigenvalues of $E(\bt) \star_{Q, \bt} -$, there exists a corresponding semi-orthogonal decomposition
    $\op{D}^b(X) = \br{C_1, \ldots, C_N},$
    where $N$ is the number of distinct eigenvalues of $E(\bt) \star_{Q, \bt} -$.
\end{conjecture}

A closely related conjecture was put forth by Sanda-Shamoto \cite{SS} using mutation systems.
In the context of birational transformations, related conjectures have been proposed by Iritani in \cite{Iri-toric} and Halpern-Leistner in \cite{HL}.
In all of these conjectures,
the semi-orthogonal decomposition of $\op{D}^b(X)$ should be compatible 
with an associated decomposition of $\nabla_{z {d\over d z}}$.
This compatibility may be seen via the asymptotics of solutions to $\nabla$ as $z \to 0.$  We will describe the expected relationship  in our specific context more precisely below.

\subsection{Semiorthogonal decompositions for flips}

In light of Conjecture~\ref{c:K1}, particular semi-orthogonal decompositions of the derived category should correspond to the quantum spectrum at certain special points $(Q, \bt) \in \cc M^{alg}$ (together with a choice of paths).  
As a first step in verifying this principle, in this paper we prove that several natural semi-orthogonal decompositions, namely those arising for projective bundles, blow-ups, and standard flips, are related to the quantum spectrum at a particular special point in $\cc M^{alg}$ in the sense of
Conjecture~\ref{c:K1}.

Consider first the case of blow-ups.  Let $\bar X$ be a smooth projective variety with a smooth codimension $r$ subvariety
$Z.$
Let $p: X \to \bar X$ be the blow-up of $\bar X$ at $Z$.
Orlov's celebrated semi-orthogonal decomposition of the blow-up \cite{Orl-blowup} 
 is 
\begin{equation*}
\label{sod-blowup}
\op{D}^b(X) = \langle \op{D}^b_{-(r-1)}(Z), \op{D}^b_{-(r-2)}(Z), \ldots, \op{D}^b_{-1}(Z), L\pi^*\op{D}^b(\bar X)\rangle,
\end{equation*}
where $\{\op{D}^b_{i}(Z)\}$ are subcaterogies of $\op{D}^b(X)$ given by certain fully-faithful 
functors from $\op{D}^b(Z)$ to $\op{D}^b(X)$.

This was later extended to standard flips by Belmans-Fu-Raedschelders \cite{BFR}, based on the work of Bondal-Orlov \cite{BO}.
The general setup is as follows (see also \cite{BFR}). 
Let $p: X \to \bar X$ be a birational morphism of smooth projective varieties which is an isomorphism away from $F=p^{-1}(Z)$ for $Z$ a smooth subvariety of $\bar X$.  
We denote the restriction of $p$ on $F$ by $\pi$. 
Suppose further that there exist vector bundles $V, V' \to Z$ of rank $r$ and $s$ respectively such that
$F \cong \PP(V)$, 
with normal bundle 
\begin{equation*}
\label{normal-bundle-assumption}
N_{F|X} \cong \pi^*(V')  \otimes\mathcal O_{\PP(V)}(-1). 
\end{equation*}
Let $\tau :\widehat X \to X$ be the blowup of $X$ along $F$.  The exceptional divisor $E$ is isomorphic to $\PP(V) \times_Z \PP(V')$, with normal bundle $\mathcal O(-1,-1).$  In the category of smooth algebraic spaces $E$ can be contracted in $\widehat X$ to obtain $\widehat X \to X'.$  If $X'$ is again a smooth projective variety then this is a blow down and $X \dasharrow X'$ is a \emph{standard flip}.  We have the following diagram:
\begin{equation}\label{e:diagram}
\begin{tikzcd}
& &E \ar[d, hook] \ar[ddll] \ar[ddrr]&& \\
&& \widehat X \ar[dl, swap, "\tau"] \ar[dr, "\tau'"] && \\
    F  \ar[ddrr, "\pi"] \ar[r, hook, "j"] &X \ar[rr, dashed] \ar[dr, "p"] &&X' \ar[dl, swap, "p'"] &F' \ar[ddll, swap, "\pi '"] \ar[l, hook', swap, "j'"] \\
    && \bar X && \\
    && Z \ar[u, hook, "\bar j"] &&
\end{tikzcd}
\end{equation}

Bondal-Orlov \cite{BO} and Belmans-Fu-Raedschelders \cite{BFR} have constructed fully faithful functors 
\begin{equation}\label{sod-functors}
\begin{dcases}
\FM: \op{D}^b(X')\to \op{D}^b(X), & E \mapsto R\tau_* L\tau'^*(E) \\
\Phi_m^X: \op{D}^b(Z)\to \op{D}^b(X), & E\mapsto j_*(\pi^*E\otimes \mathcal{O}_F(m)).
\end{dcases}
\end{equation}
Denote the subcategory $\Phi_m^X(\op{D}^b(Z))$ by $\op{D}^b_m(Z)$. The following result is proved in \cite{BFR}.
\begin{theorem}\cite{BFR}
For any integer $k$ such that $0\leq k\leq r-s$, $\op{D}^b(X)$ admits a semiorthogonal decomposition
\begin{equation}
    \label{sod-flip}
    \op{D}^b(X)=\left\langle \op{D}^b_{-k}(Z), \cdots, \op{D}^b_{-1}(Z), \FM (\op{D}^b(X')),
    \op{D}^b_0(Z), \cdots, \op{D}^b_{r-s-1-k}(Z)\right\rangle.
\end{equation}
\end{theorem}
Note that when $s=1$ \eqref{e:diagram} reduces to  a blow-up, and we may view a projective bundle as the degenerate case of \eqref{e:diagram} where $s=0$, $X = F$, and $X' = \emptyset$.
In these cases, \eqref{sod-flip} recovers the semi-orthogonal decompositions given by Orlov in \cite{Orl-blowup}.

\subsection{A decomposition of $H^*(X)$ via Gamma classes}
\label{s:RGC}

Let $\mathcal{V}$ be a vector bundle over $X$ of rank $n$ and $\{\delta_i\}_{i=1}^{n}$ be its Chern roots.
Consider the $(2 \pi \sqrt{-1})$-modified 
Chern character and Todd class of $\mathcal{V}$
\begin{equation}
\label{modified-todd}
    {\ch}(\mathcal{V}):= \sum_{i=1}^{n} e^{2\pi\sqrt{-1}\delta_j}, \;\;\; {\rm Td}(\mathcal{V}):=\prod_{j=1}^{n}
    {2\pi\sqrt{-1}\delta_j\over 1-e^{-2\pi\sqrt{-1}\delta_j}}.
\end{equation} 
By taking Hochscild homology and applying the HKR isomorphism
, the semiorthogonal decomposition \eqref{sod-flip} yields a decomposition 
in cohomology: 
$$H^*(X)=\UU^X(H^*(X'))\bigoplus\bigoplus_{m=-k}^{r-s-1-k}\Psi_m^X(H^*(Z)),$$ 
where $\UU^X$ and $\Psi_m^X$ are linear maps given by 
\begin{equation}\label{cohomology-maps}
    \begin{dcases}
    \UU^X 
    : H^*(X')\to H^*(X), &\alpha \mapsto  {\tau}_*\left( \operatorname{Td}(T\tau) {\tau'}^*(\alpha) \right)\\
        \Psi_m^X: H^*(Z)\to H^*(X), &
        \alpha\mapsto \left(2\pi\sqrt{-1}\right)^s j_*\left({\rm Ch}( \mathcal{O}_F(m)){\rm Td}(N_{F|X})^{-1} \pi^*(\alpha)\right).
    \end{dcases}
\end{equation}

Following \cite{Hos, KKP08, Iri-integral}, we consider another decomposition of $H^*(X)$  via Gamma classes that naturally arises from the perspective of mirror symmetry. 
The {\em Gamma class of $\mathcal{V}$} is given by 
\begin{equation}
\label{gamma-class-vect}
\widehat \Gamma_{\mathcal{V}} := \prod_{i=1}^{n} \Gamma(1 + \delta_i).
\end{equation}
Here $\Gamma(1 + x)$ is expressed as a power series at $x = 0$.  
It is convenient to define
\begin{equation}
    \label{gamma-class-minus}
\widehat{\Gamma}_{\mathcal{V},\pm}=\prod_{i=1}^{n}\Gamma(1\pm\delta_i).
\end{equation}
Using the Euler reflection formula 
\begin{equation}
    \label{euler-reflection}
    \Gamma(1+x) \Gamma(1-x) = \frac{2 \pi \sqrt{-1} x e^{-\pi  \sqrt{-1} x}}{1 - e^{- 2\pi  \sqrt{-1} x}}={\pi x\over \sin (\pi x)},
\end{equation}
the class
${\rm Td}(\mathcal{V})$ 
can be decomposed as
\begin{equation}
\label{todd-from-gamma}
{\rm Td}(\mathcal{V})=e^{\pi\sqrt{-1}c_1(\mathcal{V})}\widehat\Gamma_{\mathcal{V},+}\widehat\Gamma_{\mathcal{V},-}.
\end{equation}
Thus the Gamma classes can be viewed as ``asymmetric square roots" \cite{GGI} of ${\rm Td}(\mathcal{V})$.
Define   
$$\widehat\Gamma_X:=\widehat\Gamma_{TX}$$
to be the {\em Gamma class of $X$}.
We may also view $\widehat{\Gamma}_X$ as an operator $\widehat{\Gamma}_X\cup \in{\rm End}(H^*(X))$. Using the maps defined in \eqref{cohomology-maps}, we consider the decomposition 
\begin{equation}
\label{gamma-decomposition}
H^*(X)=\widehat{\Gamma}_X\UU^X(H^*(X'))\bigoplus\bigoplus_{m=-k}^{r-s-1-k}\widehat{\Gamma}_X\Psi_m^X(H^*(Z)).
\end{equation}
The main goal of this paper is to study the asymptotics of certain flat sections of the quantum connection defined using the decomposition \eqref{gamma-decomposition}.





\subsection{Quantum spectrum of $c_1(X)\star_q-$}
Following the Hukuhara-Levelt-Turrittin Theorem, we can study the asymptotic behavior of the flat sections of the connection \eqref{quantum-connection} when $z$ approaches the irregular singularity through the quantum spectrum of $E(\bt) \star_{Q, \bt}- $.
We will prove that the semi-orthogonal decomposition in \eqref{sod-flip} is closely connected to the quantum spectrum at a particular specialization of $(Q, \bt)$.

Recall that $\pi : F \to Z$ is the exceptional locus of the flip $X \dasharrow X'.$  Let $L$ denote a line in $F $ contracted by $\pi$. 
Let us restrict our quantum product as follows.  For $d \in NE(X)_\ZZ$, set $Q^d = 1$ if $d = k[L]$ for some $k \geq 0$, and $Q^d = 0$ otherwise.  Further, specialize $\bt \in H^*(X)$ to $\ln(q) c_1(X)$.
This is proven in Corollary~\ref{c:assoc1} to yield a well-defined associative product $\star_q$ which we call the \emph{extremal quantum product}, as the quantum corrections involve only those curves in the extremal ray  of the cone of curves consisting of those curves contracted by $p: X \to \bar X$.  With this specialization, the operator $E(\bt) \star_{Q, \bt} - $ simplifies to $c_1(X) \star_q - .$

The quantum spectrum of $c_1(X)\star_q -$ is given as follows.
\begin{theorem}\label{qspec}
Assume that $r-s > 1$. 
After specializing as above,
    the operator $c_1(X)\star_q -$ on  $QH^*(X)$ has an eigenvalue of zero of multiplicity $\rank(H^*(X'))$, and $r-s$ nonzero eigenvalues given by
    \begin{equation}
\label{evalue-m-th}
\lambda_k(q):=(r-s)e^{-\pi\sqrt{-1}{2k+s\over r-s}} q
\end{equation}
     for $0\leq k < r-s$, each of multiplicity $\rank(H^*(Z))$. 
\end{theorem}
This result was shown by Kontsevich in the case of blow-ups (s=1) using a dimension argument.  

\subsection{Asymptotic classes}
Next, inspired by Galkin-Golyshev-Iritani~\cite{GGI} we study asymptotic classes with respect to the extremal quantum cohomology of $p:X \to \bar X$, yielding a modification of Galkin-Golyshev-Iritani's Gamma Conjecture II \cite{GGI} and Sanda-Shamoto's Dubrovin type conjecture \cite{SS}
to this new setting.

Using the $S$-operators $S^X(Q, \bt, z)$ (defined in \eqref{S-operator}) in Gromov-Witten theory, the quantum connection \eqref{quantum-connection} admits multi-valued flat sections (see Proposition \eqref{Phi})
\begin{equation*}\label{e:fs}
\{\Phi^X(Q, \bt, z)(\alpha):=(2\pi)^{-{\dim X\over 2}}S^X(Q,\bt, z)z^{-\mu}z^{c_1(X)\cup} \alpha \mid \alpha\in H^*(X)\}.\end{equation*}  
We focus on the asymptotic behavior of these flat sections near the irregular singularity $z=0$.
Following \cite{GGI}, we introduce asymptotic classes as below.  
\begin{definition}
\label{def-asymptotic}
Choose a fixed value of $(Q, \bt)$ for which $\star_{Q,\bt}$ is convergent. 
Let $\lambda\in \mathbb{C}$ be an eigenvalue of $E\star_{Q, \bt}$.

If $\lambda\neq0$, we say $\alpha\in H^*(X)$ is a strong asymptotic class with respect to $\lambda$ and argument $\arg(\lambda)$ if there exists some $m\in\mathbb{R}$ such that
\begin{equation}
\label{asymp-class-behavior}
|\!|
e^{\lambda \over z} \Phi^X(Q, \bt, z)(\alpha)
|\!|
= O(|z|^{-m})
\end{equation}
as $z\to 0$ with $\arg(z)=\arg(\lambda)$.

If $\lambda=0$, we say $\alpha\in H^*(X)$ is a strong tame asymptotic class with respect to argument $\theta$ if \eqref{asymp-class-behavior} holds as $z\to 0$ with $\arg(z)=\theta.$
\end{definition}

Our main results relate the quantum spectrum and asymptotic classes of $X$ for a standard flip $X \dasharrow X'$.  
The results also hold in the cases $s=0$ and $s=1$, corresponding to projective bundles and blow-ups respectively.
\begin{theorem}\label{main-theorem}
For the standard flip $X \dasharrow X'$, the decomposition \eqref{gamma-decomposition} of $H^*(X)$ is a decomposition of strong asymptotic classes with respect to the quantum spectrum of $c_1(X)\star_q -$. 
More explicitly:
\begin{enumerate}
    \item the elements in $\widehat{\Gamma}_X\Psi_m^X(H^*(Z))$ are strong asymptotic classes with respect to the eigenvalue $\lambda_m(q)$ and $\arg(z/q)={-2m-s\over r-s}\pi$;
    \item the elements in $\widehat{\Gamma}_X\UU^X(H^*(X'))$ are strong tame asymptotic classes with  $\arg(z/q)={1-s\over r-s}\pi$.
\end{enumerate}
\end{theorem}
In particular, we can apply this to the semi-orthogonal decomposition of $\op{D}^b(X)$ given in \eqref{sod-flip}. Recall that $\FM$ and $\Phi_m^X$ are functors defined in \eqref{sod-functors}. 
We have
\begin{corollary}
The semi-orthogonal decomposition of $\op{D}^b(X)$  in \eqref{sod-flip}  is compatible with the extremal quantum spectrum of $X$ in the following sense:  
\begin{enumerate}
\item \label{i11}
   for $E \in \op{D}^b(Z)$, $\widehat \Gamma_X \operatorname{Ch}(\Phi_m^X(E))$  is a strong asymptotic class with respect to $\lambda_m(q)$ 
and $\arg(z/q)=\frac{-2m-s}{r-s}\pi;$
\item \label{i22}
for $E' \in \op{D}^b(X')$, $\widehat \Gamma_X \operatorname{Ch}(\FM (E'))$ is a strong tame asymptotic class 
with respect to 
$\arg(z/q)=\frac{1-s}{r-s}\pi.$ \end{enumerate}
\end{corollary}

\begin{remark}
    Note that the order of components in the semi-orthogonal decomposition of $\op{D}^b(X)$ matches the order (moving clockwise) of the rays along which we take asymptotic expansions of the associated asymptotic classes.
    \begin{equation*}
        \begin{array}{|l|c|c|c|c|c|c|c|}
        \hline
     & \op{D}^b_{-k}(Z)  & \cdots &\op{D}^b_{-1}(Z)&\FM(\op{D}^b(X'))& \op{D}^b_0(Z)&  \cdots    & \op{D}^b_{r-s-1-k}(Z) \\ \hline
      \arg\left({z\over q}\right)      &  \frac{2k-s}{r-s}\pi & \cdots 
        &   \frac{2-s}{r-s}\pi       &
          \frac{1-s}{r-s}\pi & 
           \frac{-s}{r-s}\pi & \cdots &  \frac{-2(r-s-1-k)-s}{r-s}\pi
           \\ \hline
        \end{array}
         \end{equation*}

If $k$ satisfies ${1\over 4}(r-s-6)<k<{3\over 4}(r-s+2)$, according to \eqref{arg-z-expoenential} and \eqref{argument-tame}, we can find a common sector such that the asymptotic expansions in Proposition \ref{exponential-asymptotic-local-model} and Proposition \ref{mellin-barnes} hold.
\end{remark}

\begin{remark}
[Related works]
The quantum cohomology of a blowup of a smooth projective variety at a point has been studied by Bayer \cite{Bayer} and Milanov-Xia \cite{MX}. 
  The quantum cohomology of toric blow-ups and flips has been studied by Gonzalez-Woodward in \cite{GW}, Acosta and the second author in \cite{AS}, and by Iritani in  \cite{Iri-toric}.  
 The quantum ring structure of simple flips has been studied by Lee-Lin-Wang for the local model in \cite{LLW21} and a simple $(2, 1)$ flip in \cite{LLW17}.  
 A decomposition of the quantum $D$-module for projective bundles and for the blow-up of a projective variety along a smooth center has been obtained by Iritani-Koto \cite{IK} and Iritani \cite{Iri-blowup}, respectively.
\end{remark}

\subsection{Plan of the paper}

In Section~\ref{s:setup} we review the basic definitions and define the extremal quantum cohomology.  

Sections \ref{s:local}-\ref{s:spectrum} are concerned primarily with proving Theorem~\ref{qspec}.
In Section~\ref{s:local} we define the local model and show that there is a ring homomorphism from the extremal quantum cohomology of $X$ to that of the local model.  In Section~\ref{s:relquant} we consider fiber bundles, and compute the extremal quantum cohomology of the local model. In Section~\ref{s:spectrum} we compute the extremal quantum spectrum.

In Sections~\ref{s:projbund}-\ref{s:zflip} we compute the (weak) asymptotic classes in special cases.  In Section~\ref{s:projbund} we compute the asymptotic classes for a projective bundle, while Sections~\ref{s:nzflip} and~\ref{s:zflip} compute the  asymptotic classes with respect to the nonzero and zero eigenvalues respectively for the local model.

Finally, in Section~\ref{s:reductions} we apply a sequence of reductions to compute the asymptotic classes for general $X$.  The main technique is a deformation to the normal cone of  Diagram~\eqref{e:diagram}.

\subsection{Acknowledgement}
We would like to thank Hiroshi Iritani for sharing valuable insight and patiently answering many questions on the Gamma conjectures and related topics.  
We also thank Margaret Cheney, Lie Fu, Maxim Kontsevich, and Y.-P. Lee for valuable conversations. 

Y. Shen is partially supported by a Simons Foundation Travel Grant. 
M. Shoemaker is partially supported by Simons Foundation Travel Grant 958189.

\section{Quantum connection and extremal quantum cohomology}\label{s:setup}

\subsection{Quantum connection and $J$-functions}
Let $X$ be a smooth projective variety. Let $H^*(X) := H^*(X; \CC)$ denote cohomology of $X$ with complex coefficients.
Choose a basis $\{\phi_i\}$ of homogeneous elements of $H^{*}(X)$ and let
$\bt = \sum_{i=1}^n t^i \phi_i$ be a formal linear combination of this basis.  Define $\CC[[Q]]$ to be the \emph{Novikov Ring}
$$\CC[[Q]] = \left\{ \sum_{d \in NE(X)_\ZZ} a_d Q^d \; | \;  a_d \in \CC \right\},$$
where $Q^{d} \cdot Q^{d'} = Q^{d + d'}$.  
Choose a basis $d_1, \ldots, d_r$ of $H_2(X, \ZZ)$ such that for an ample line bundle $L \to X$, $L\cdot d_i \geq 0$ for all $i$.  Then we have an inclusion 
$\CC[[Q]] \hookrightarrow \CC[[Q_1, \ldots, Q_r]]$
given by $Q^d \mapsto \prod_{j=1}^rQ_j^{c^j}$ if $d = \sum_{j=1}^r c^j d_j.$
Let $\CC[[Q, \bt]] = \CC[[Q]] \otimes_\CC \CC[[ t^1, \ldots, t^n]]$.

Let $X_{0,k,d}$ be the moduli stack of genus-zero, $k$-pointed stable maps to $X$ of degree $d$ with the virtual fundamental class $[X_{0,k,d}]^{\rm vir}$ and ${\rm ev_i}$ be the evaluation map at the $i$-th marked point.
Let $\br{\cdot, \cdot}^X$ be the Poincar\'e paring of $X$.
The quantum product $$\star_{Q, \bt}: H^*(X) \times H^*(X) \to H^*(X)\otimes_\CC\CC[[Q, \bt]]$$ for $X$ is defined by 
\begin{equation}\label{d:qprod} \br{\alpha \star_{Q, \bt} \beta, \gamma}^X = \brr{\alpha, \beta, \gamma}^X(Q, \bt), \end{equation}
where $\brr{\alpha, \beta, \gamma}^X(Q, \bt)$
is the  generating function of genus zero Gromov-Witten invariants 
\begin{align*}
    \brr{\alpha, \beta, \gamma}^X(Q, \bt) 
    =&  \sum_{d} \sum_{n \geq 0} \frac{Q^{d}}{n!} \br{\alpha, \beta, \gamma, \bt, \ldots, \bt}_{0,n+3, d}^X\\
    =&\sum_{d} \sum_{n \geq 0} \frac{Q^{d}}{n!} \int_{[X_{0,n+3,d}]^{\rm vir}}{\rm ev}_1^*(\alpha){\rm ev}_2^*(\beta)
{\rm ev}_3^*(\gamma)\prod_{i=4}^{n+3}{\rm ev}_i^*(\bt).
\end{align*}

We consider the Dubrovin connection on the trivial $H^*(X)$-bundle over the base $\op{Spec}\left( \CC[[Q, \bt]]\right)\times \mathbb{P}^1.$
Let $z\in\mathbb{P}^1$. 
The Dubrovin connection $\nabla$ takes the form
\begin{equation}\label{e:dubcon}
\begin{dcases}
&\nabla_{\partial \over \partial t_i}={\partial \over \partial t_i}+{1\over z}\phi_i\star_{Q, \bt}\\
&\nabla_{Q_j{\partial \over \partial Q_j}}=Q_j{\partial \over \partial Q_j}+{1\over z}\op{PD}(\beta_j) \star_{Q,\bt}\\
&\nabla_{z{\partial\over \partial z}}=z{\partial\over \partial z}+\mu-{1\over z}E(\bt)\star_{Q, \bt} .
\end{dcases}
\end{equation}
Here $\op{PD}$ denotes the Poincar\'e dual,  $E(Q, \bt)$ is the Euler vector field
\begin{equation}
\label{euler-vector}
E(Q, \bt)  =  c_1(X) + \sum_{i=1}^n \left(1 - \frac{\deg(\phi_i)}{2}\right)t^i \phi_i
\end{equation}
and $\mu$ is the Hodge grading operator: 
\begin{equation}
    \label{hodge-grading}
\mu(\alpha) = \left(p - \frac{\dim(X)}{2}\right)\alpha, \quad \forall \quad \alpha \in H^{2p}(X).
\end{equation}
In the above, the Novikov variables $Q$ is viewed as a parameter.
One can further extend this to a logarithmic connection over $\op{Spec}\left(\CC[[Q, \bt]]\right)\times \mathbb{P}^1,$ but due to the divisor equation,  
$\nabla_{Q{\partial \over \partial Q}}$ is determined by $\nabla_{\partial \over  \partial t_i}$ in the $H^2(X)$ directions.

Let $S^X(Q, \bt, z)$ be the  operator given by
\begin{equation} \label{S-operator}
\br{S^X(Q, \bt,z) \alpha, \beta}^X = \br{\alpha, \beta}^X + \sum_d \sum_{n \geq 0} \frac{Q^d}{n!}\br{\frac{\alpha}{-z - \psi_1}, \beta, \bt, \ldots, \bt}^X_{0, n+2, d}.
\end{equation}
Let $\rho:=c_1(X)\cup -$ on $H^*(X)$.
The following result is well known.
\begin{prop}
\cite{Dub, Giv, Pand, Iri-integral, GGI}
\label{Phi}
The connection \eqref{e:dubcon} admits multi-valued flat sections 
\begin{equation}\label{flat-sections}
\Phi^X(Q, \bt, z)(\alpha)=(2\pi)^{-{\dim X\over 2}}S^X(Q,\bt, z)z^{-\mu}z^{\rho} \alpha, \quad \text{for all } \alpha\in H^*(X).
\end{equation}
\end{prop}

\begin{remark}
    If $X$ is endowed with the action of a torus $\TT$, then one may use equivariant Gromov--Witten theory to define $\TT$-equivariant versions of the quantum product, Dubrovin connection, and all operators given in this section.  We will make use of this theory at certain points in paper, in which case we will use a $\TT$-subscript to make clear that we are working equivariantly (e.g. the equivariant solution matrix of Proposition~\ref{Phi}  will be denoted by $\Phi^X_\TT$).  For the foundations of equivariant Gromov--Witten theory and the application of the Atiyah--Bott localization formula therein, we refer the reader to \cite{GP}.
\end{remark}

\subsection{$J$-function and weak asymptotic classes}
The operator $S^X(Q, \bt, z)$ is invertible and the {\em $J$-function} of $X$ is defined by
\begin{equation}\label{J-function-def}
J^X(Q, \bt, z):=z S^X(Q,\bt,z)^{-1}\one.
\end{equation}

We can study the term $\left\langle \one, \Phi^X(Q, \bt, z)(\alpha)\right\rangle^X$ of the flat sections using the $J$-function. 
 It is convenient to multiply this term by a factor of $(2\pi z)^{\dim X\over 2}$ and rewrite it using a modified $J$-function 
\begin{equation}\label{modified-pi}
    \wtJ^X(Q, \bt, -z):=z^{\rho}z^{\dim X\over 2}z^{\mu}{J^X(Q, \bt, -z)\over -z}.
\end{equation}
\begin{lemma}
We have 
\begin{equation}
    \label{leading-term}
(2\pi z)^{\dim X\over 2}\left\langle \one, \Phi^X(Q, \bt, z)(\alpha)\right\rangle^X
=\left\langle \wtJ^X(Q, \bt, -z),  \alpha\right\rangle^X.
\end{equation}
\end{lemma}
\begin{proof}
Using properties of $\mu$, $S^X(Q, \bt, z)$, we have 
    \begin{align*}
  (2\pi z)^{\dim X\over 2}\left\langle \one, \Phi^X(Q, \bt, z)(\alpha)\right\rangle^X
    &=z^{\dim X\over 2}\left\langle \one, S^X(Q, \bt, z)z^{-\mu}z^{\rho} \alpha\right\rangle^X\\
    &=z^{\dim X\over 2}\left\langle S^X(Q, \bt, -z)^{-1}\one, z^{-\mu}z^{\rho} \alpha\right\rangle^X\\
    &=z^{\dim X\over 2}\left\langle z^{\mu}S^X(Q, \bt, -z)^{-1}\one, z^{\rho} \alpha\right\rangle^X\\
    &=\left\langle z^{\rho}z^{\dim X\over 2}z^{\mu}S^X(Q, \bt,-z)^{-1}\one, \alpha\right\rangle^X\\
    &=\left\langle \wtJ^X(Q, \bt, -z), \alpha\right\rangle^X.
\end{align*}
\end{proof}

In particular, for $E\in K^0(X)$, we have the following formula for the {\em quantum cohomological central charge} of $E$ (defined in \cite{Hos, Iri-integral})
\begin{align}
Z(E)(Q, \bt)&:=(2\pi z)^{\dim X\over 2}\langle \one, \Phi^X(Q, \bt, z)(\widehat{\Gamma}_X\cup {\rm Ch}(E))\rangle^X  \label{central-charge-def}\\
&=\left\langle\wtJ^X(Q, \bt, -z),\widehat{\Gamma}_X{\rm Ch}(E)\right\rangle^X. \label{central-charge-modified-J}
\end{align}

Before Section \ref{s:reductions}, we will only focus on the asymptotic behavior of the terms in \eqref{leading-term} by introducing the concept of weak asymptotic classes. 
\begin{definition}\label{weak-asymp-def}
For fixed $(Q, \bt),$ we say $\alpha\in H^*(X)$ is a \emph{weak asymptotic class} with respect to $\lambda\in \mathbb{C}^*$ and argument $\arg(\lambda)$ if there exists some $m\in\mathbb{R}$ such that
\begin{equation}
\label{weak-asymp-class-behavior}
|\!|
e^{\lambda \over z} \langle\one, \Phi^X(z)(\alpha)\rangle
|\!|
= O(|z|^{-m})
\end{equation}
as $z\to 0$ 
with $\arg(z)=\arg(\lambda)$.

If $\lambda=0,$ $\alpha\in H^*(X)$ is called a weak tame asymptotic class with respect to argument $\theta$ if \eqref{weak-asymp-class-behavior} holds as $z\to 0$ 
with $\arg(z)=\theta$.
\end{definition}

\begin{remark}
Here we don't assume $\lambda$ to be an eigenvalue of the quantum multiplication. This is convenient when we discuss the asymptotic behavior in cases when the quantum spectrum is not fully known \cite{SZ}. 
\end{remark}

\subsection{Extremal quantum cohomology}\label{s:exc1}

In this section and going forward we focus on the special case that $X$ and $X'$ are related by a standard flip as in Diagram~\eqref{e:diagram}.  We will further specialize the quantum product to involve only the \emph{extremal curve classes}.    We will call this specialized product the extremal quantum cohomology.
More precisely, recall that $F$ is the exceptional locus of $p: X \to \bar X$.  Let $L \subset F$ be a line contained in a fiber of $p: X \to \bar X$.  In the quantum product \eqref{d:qprod}, we will set $Q^d = 0$ unless $d = k[L]$.
 We first verify that with this specialization we still have an associative quantum product.

\begin{lemma}\label{l:fiberclass}
    If a curve $D$ is of class $k[L]$, where $L$ is a line in $F$ contracted by $p$, then $D$ lies in a fiber of $\pi: F \to Z$.
\end{lemma}

\begin{proof}
    Suppose not, then $\pi_*(D)$ is an effective curve of degree zero, but this is a contradiction by Kleiman's criterion.
\end{proof}
\begin{prop}\label{p:assoc1}
The ray $\RR_{\geq 0} [L]$ is extremal in $NE(X) \subset H_2(X, \RR)$. 
\end{prop}
\begin{proof}
We claim that for $d$ in the span of $[L]$, if  $d=d_1 + d_2$ is a sum of two classes, with $d_1, d_2 \in \overline{NE}(X)$, then $d_1$ and $d_2$ are both in the span of $[L].$  
    
Let $K$ denote the kernel of $H_2(X, \RR) \to H_2(\bar X, \RR).$  We first argue that given $d_1, d_2 \in \overline{NE}(X)$ with $d_1 + d_2 \in K$, then both $d_1$ and $d_2$ lie in $K$.  
To see the claim, assume the contrary, then without loss of generality, $p_*(d_1)$ is a nonzero effective curve class in $H_2(\bar X, \RR)$.  Since $p_*(d_1 + d_2) = 0$, this means $p_*(d_2) = -p_*(d_1)$.  This is a contradiction by Kleiman’s criterion, which states that because $p_*(d_1)$ and $p_*(d_2)$ are nonzero in $\overline{NE}(\bar X)$, they must both pair positively with an ample divisor.

We may therefore assume that $d_1, d_2  \in K.$  Let $U = \bar X \setminus Z  = X \setminus F.$  We have the following diagram
\[\begin{tikzcd}
U \ar[r, "\phi"] \ar[d, "="] & X \ar[d, "p"] \\
U \ar[r, "\bar \phi"] & \bar X
\end{tikzcd}
\]
Then $\phi^* = \bar \phi^* \circ p_*$, thus $K \subset \ker(\phi^*).$
Via the exact sequence
$$A_1(F) \to A_1(X) \to A_1(U) \to 0,$$
we see that $d_1, d_2$ lie in the image of $H_2(F, \RR).$

Let $\widehat d_1$ and $\widehat d_2 $ be  classes in $\overline{NE}(F) \subset H_2(F, \RR)$ that map to $d_1$ and $d_2$ under $j_*$ for $j: F\to X$.  Note that 
$$H_2(F, \RR) = H_2(Z, \RR) \oplus \RR[L]$$
where the inclusion $H_2(Z, \RR) \to H_2(F, \RR)$ is the map $\alpha \mapsto H^{r-1} \cap \operatorname{PD}^{-1}(\pi^*\operatorname{PD}(\alpha))$, with $\operatorname{PD}$ denoting the Poincar\'e dual.
Using this splitting, write $\widehat d_i$ as
\begin{equation}\label{e:di}
    \widehat d_i = (H^{r-1} \cap \operatorname{PD}^{-1}(\pi^*\operatorname{PD}(\alpha_i)))+ n_i[L].
\end{equation}
Pushing forward via $\bar j \circ \pi$, we get that $\bar j_*(\alpha_1) = \bar j_*\pi_*(\widehat d_1)$ and $\bar j_*(\alpha_2) = \bar j_*\pi_*(\widehat d_2)$ satisfy $\bar j_*(\alpha_1) = - \bar j_*(\alpha_2)$.  Therefore 
$$\deg(\bar j_*(\alpha_1)) = \deg(\bar j_*(\alpha_2)) = 0.$$  
Then $\alpha_1$ and $\alpha_2$ are  classes in $\overline{NE}(Z)$ of degree zero, and therefore are both zero.  By \eqref{e:di}, we conclude that $d_1$ and $d_2$ both lie in the span of $[L]$, proving the claim.
\end{proof}

\begin{definition}\label{d:exq}
Let $\CC[[q]]$ be the ring of formal power series in $q$.  Define the  specialization of variables from the Novikov ring to the power series ring
$\CC[[Q]] \to \CC[[q]]$
$$Q^d \mapsto \left\{ \begin{array}{l}
q^{k(r-s)} \text{ if } d = k[L] \text{ for some } k \in \ZZ \\
0 \text{ otherwise.}
\end{array}\right.
$$
By Proposition~\ref{p:assoc1} this gives a well-defined ring homomorphism.
Define the (big) \emph{extremal quantum cohomology of $X$} to be the quantum cohomology ring after this restriction of variables.  Denote this product by $\star_{q, \bt}$.  If we further specialize to $\bt = 0$, we obtain the small extremal quantum product.  We denote it by $\star_q$, and the ring by $QH^*_{ext}(X).$
\end{definition}
\begin{corollary}\label{c:assoc1}
The extremal quantum product is associative.
\end{corollary}
\begin{proof}
From Proposition~\ref{p:assoc1}, one  sees that for $d \in \operatorname{span}([L])$, the $Q^{d}$-coefficient of  $\alpha \star_{Q, \bt} (\beta \star_{Q, \bt} \gamma)$ and $(\alpha \star_{Q, \bt} \beta) \star_{Q, \bt} \gamma$  are each a sum over  only those $d_1 + d_2 = d$  such that $d_1, d_2 \in \operatorname{span}([L])$.  Thus the WDVV equation in extremal quantum cohomology follows from the usual WDVV equation for $X$.
\end{proof}

\begin{lemma}\label{l:conv} 
Only finitely many Gromov--Witten invariants contribute to the extremal quantum product $\alpha \star_{q} \beta$.
\end{lemma}

\begin{proof}
The virtual dimension of 
$\sMbar_{g,n}(X, d)$ is equal to 
$$(1-g)(\dim X - 3) + n + \br{c_1(X), d}.  $$
For $d = k[L]$ with $k \geq 0$, we have $\br{c_1(X), d} = (r-s) \cdot k$.  By construction we have $r>s$, so for fixed $g, n$, there are only finitely many choices of $k$ such that the virtual dimension of $\sMbar_{g,n}(X, k[L])$ is less than or equal to $n \cdot \dim(X)$ - the maximal possible complex degree of the product of insertions.
\end{proof}

By the divisor equation, the following specializations of the quantum product are equal
$$ \alpha \star_{Q, \bt}\beta\vline_{\substack{\; Q^d = q^{k(r-s)} \text{ if } d = k[L]\\ \;Q^d = 0 \text{ otherwise}\\ \; \bt = \bt}} = \alpha \star_{Q, \bt}\beta\vline_{\substack{\;Q^d = 1 \text{ if } d = k[L]\\ \;Q^d = 0 \text{ otherwise}\\ \; \bt = \bt + \ln(q)c_1(X)}}.
$$

It is this second specialization that we will use to specialize the $S$-operator. 

By Corollary~\ref{c:assoc1} and Lemma~\ref{l:conv},  for each fixed $q\neq0$ we have a meromorphic flat connection over the base $\mathbb{P}^1$, 
denoted by
\begin{equation}
\label{mero-connection}
\nabla_{z{\partial\over \partial z}}=z{\partial\over \partial z}-{1\over z}c_1(X)\star_{q}+\mu.
\end{equation}
This meromorphic connection has a regular singular point at $z=\infty$ and an irregular singular point at $z=0$.

We will specialize the $S$-operator and $J$-function as well.  Let 
 \begin{align}
 \label{modified-pi-extremal2}
 S^X(q, \bt, z) &:= S^X(Q,\bt,z)\vline_{\substack{\;Q^d = 1 \text{ if } d = k[L]\\ \;Q^d = 0 \text{ otherwise}\\ \;\bt = \bt + \ln(q)c_1(X)}} 
 \\ \nonumber
    \wtJ^X(q, \bt, -z)&:=
    \wtJ^X(Q, \bt, -z)\vline_{\substack{\; Q^d = 1 \text{ if } d = k[L]\\ \;Q^d = 0 \text{ otherwise}\\ \;\bt = \bt + \ln(q)c_1(X)}}
   \end{align}

If we further restrict $\bt=0$, we denote it simply by 
\begin{align}\label{modified-pi-extremal}
S^X(q, z) &:= S^X(q,0,z) \\ \nonumber
    \wtJ^X(q, -z) &:= \wtJ^X(q, 0, -z),
\end{align}

and the extremal central charges by $Z(E)(q)=Z(E)(q, 0)$ for $E\in K^0(X).$

\section{The local model}\label{s:local}
In this section we show that after restricting the Novikov parameter as in Definition~\ref{d:exq}, the quantum product of $X$ is closely related to that of a local model.

\subsection{Defining the local model}

Let $Z$ be a smooth projective variety with $V, V' \to Z$ vector bundles of rank $r,s$ respectively, with $r>s$.  Let $\CC^*$ act on $\tot(V \oplus V')$ by scaling with weight 1 on $V$ and weight $-1$ on $V'.$  Define $T, T'$ to be the GIT quotients associated to this action with characters of weights $+1$ and $-1$ respectively.  
More explicitly,
$T$ is the total space of 
$$V'(-1):=\pi^*V' \otimes \mathcal O_{\PP(V)}(-1),$$ 
and $T'$ is the total space of 
$${\pi '}^*V \otimes \mathcal O_{\PP(V ')}(-1).$$  
The GIT presentation yields a rational map $$T \dasharrow T'.$$

We refer to $T \dasharrow T'$ as the \emph{local model} of a standard flip.  Note that 
if $X \dasharrow X'$ is as in \eqref{e:diagram},
then $T$ may be identified with $\tot(N_{F|X})$ and $T'$ with $\tot(N_{F'|X'})$.  
\begin{equation}
\begin{tikzcd}
    T  \ar[rr, bend left, "i"] \ar[r, "k"] &  \ar[l, bend left, "j_T"] F \ar[r, "j"] & X
\end{tikzcd}
\end{equation}
where $k$ is the projection.
Note that $k^*$ and $j_T^*$ are inverse to each other.

\subsection{Comparing the virtual class}

\begin{prop}\label{fiber-moduli}
    For $d = k[L]$ with $L$ a line contracted by $p$ and $k>0$, the moduli space $\sMbar_{g,n}(X, d)$ is canonically isomorphic to $\sMbar_{g,n}(F, d)$.  The latter is a fiber bundle over $Z$, with fiber isomorphic to  $\sMbar_{g,n}(\PP^{r-1}, k)$.
\end{prop}
\begin{proof}
    
Suppose we have a map $f: C \to X $ of degree $d$.
By  Lemma~\ref{l:fiberclass},  the image of $f$ is contained in a fiber  of $F \to Z$.  This proves the first claim.  For a given $z \in Z$, choose an isomorphism $\pi^{-1}(z) \cong \PP^{r-1}.$  With this isomorphism, the subspace of $\sMbar_{g,n}(F, d)$ consisting of maps to $\pi^{-1}(z)$ is identified with $\sMbar_{g,n}(\PP^{r-1}, k)$.
\end{proof}

\begin{corollary}
    Under the identification
    $\sMbar_{0,n}(X, d) = \sMbar_{0,n}(F, d)$ from Proposition \ref{fiber-moduli}, the virtual class 
    $$[\sMbar_{0,n}(X, d)]^{vir}=e(R^1 \pi_{\cc C *} f^* N_{F|X}) = e(R^1 \pi_{\cc C *} f^* V'(-1)).$$
    
\end{corollary}
\begin{proof}
    By Proposition \ref{fiber-moduli}, the moduli space is a fiber bundle over $Z$ with smooth fiber, and therefore is smooth.  It's virtual class is the Euler class of the obstruction bundle, which in genus zero is  $R^1 \pi_{\cc C *} f^* N_{F|X}$, where
    \[
    \begin{tikzcd}
        \cc C \ar[d, "\pi_{\cc C}"] \ar[r, "f"] & X \\
        \sMbar_{0,n}(X, d) & 
    \end{tikzcd}
    \]
    is the universal map.
\end{proof}

\begin{theorem}\label{t:func}
Identifying $T = N_{F|X}$ as in the previous section, 
we have a ring homomorphism $$i^*: QH^*_{ext}(X) \to QH^*_{ext}(T).$$ 
\end{theorem}

\begin{proof}
By our assumptions in the previous section, the normal bundle $N_{F|X}$ is isomorphic to $V'(-1).$
For $d = k[L]$ with $k>0$, we have the following commutative diagram:
\[\begin{tikzcd}
\sMbar_{0,3}(T, d) \ar[r, "\cong"] \ar[d, "\rm ev_i^{T}"]& \sMbar_{0,3}(\PP(V), d) \ar[r, "\cong"] \ar[d, "\rm ev_i^{\PP(V)}"] & \sMbar_{0,3}(X, d) \ar[d, "\rm ev_i^{X}"] \\
T \ar[r, bend right = 10, swap, "k"] & \ar[l, bend right = 10, swap, "j_T"] \PP(V) \ar[r, "j"] & X
\end{tikzcd} \]
where $q$ is the projection and $j_T$ is the zero section.

Recall that $i: T \to X$ is the composition $j \circ k$.  Since the pullback commutes with cup product, to prove the theorem we must show that it commutes with the quantum corrections.  This will follow from the fact that for $d\neq 0$, $\sMbar_{0,3}(T, d)$ and $\sMbar_{0,3}(X, d)$ are canonically isomorphic, and their virtual classes are identified.
We have: 
\begin{align*}
&i^* (\rm ev_3^{X})_*\left([\sMbar_{0,3}(X, d)]^{vir} \cap \left(\rm ev_1^{X, *} \alpha \cup \rm ev_2^{X, *} \beta \right) \right) \\
&=k^* j^* j_*     (\rm ev_3^{\PP(V)})_*\left( [\sMbar_{0,3}(X, d)]^{vir}\cap \left(\rm ev_1^{\PP(V), *} j^* \alpha \cup \rm ev_2^{\PP(V), *} j^* \beta \right) \right) \\
&=k^*  \left( e(V'(-1)) \cup    (\rm ev_3^{\PP(V)})_*\left( [\sMbar_{0,3}(X, d)]^{vir}\cap \left(\rm ev_1^{\PP(V), *} j^* \alpha \cup \rm ev_2^{\PP(V), *} j^* \beta \right) \right) \right) \\
&=j_{T *} (\rm ev_3^{\PP(V)})_*\left( [\sMbar_{0,3}(X, d)]^{vir}\cap \left(\rm ev_1^{T, *}k^* j^* \alpha \cup \rm ev_2^{T, *}k^* j^* \beta \right)\right)
\\
&=(\rm ev_3^{T})_*\left( [\sMbar_{0,3}(T, d)]^{vir}\cap \left(\rm ev_1^{T, *}k^* j^* \alpha \cup \rm ev_2^{T, *}k^* j^* \beta \right)\right)
\end{align*}
The first line is the $q^d$ coefficient of $i^*(\alpha \star_q \beta)$ and the last line is the $q^d$ coefficient of $(i^*\alpha) \star_q (i^*\beta)$.
\end{proof}

\section{Extremal quantum cohomology of fiber bundles}\label{s:relquant}
In this section we prove a number of facts about what we again refer to as the extremal quantum cohomology, this time  of fiber bundles.  In the case of the local model $T$, the two notions coincide.
\subsection{Extremal quantum cohomology}
Given a fiber bundle $p: P_Z \to Z$, we consider the quantum product $\star_{Q, \bt}$ on $H^*(P_Z)$ involving only those degrees for which the map is forced to lie in a fiber, i.e. we consider degrees lying in the intersection of the effective cone with the kernel of 
$$p_*: H_2(P_Z, \ZZ) \to H_2(Z, \ZZ).$$  This is closely related to the family quantum cohomology recently constructed in 
\cite{BDOP}.

\begin{prop}
    The extremal quantum cohomology $QH^*_{ext}(P_Z)$ is an associative ring.
\end{prop}
\begin{proof}
Let $K=\ker(p_*) \subset H_2(P_Z, \ZZ)$
It suffices to show that if $d_1, d_2$ are effective curve classes such that $d_1 + d_2 \in \ker(p_*)$, then both $d_1$ and $d_2$ lie in $K$.  The argument follows (the second paragraph of) the proof of Proposition~\ref{p:assoc1}.
\end{proof}

With this specialized quantum product, the pullback along a fiber diagram defines a ring homomorphism.
Suppose that $\phi: Y \to Z$ is a map between smooth projective varieties.   
Consider the Cartesian square
\[
\begin{tikzcd}
    P_Y \ar[d] \ar[r, "\tilde \phi"] & P_Z \ar[d]\\
    Y \ar[r, "\phi"] & Z
\end{tikzcd}
\]
where  $P_Z \to Z$ is a fiber bundle.

\begin{theorem}\label{t:homom}
Given the setting above,
the map
$\tilde \phi^*: QH^*_{ext}(P_Z) \to QH^*_{ext}(P_Y)$
is a ring homomorphism.
\end{theorem}
\begin{proof}

Fix a degree $d \in K.$
By the same argument as in Lemma~\ref{l:fiberclass}, the image of any map $f: C \to X$ of degree $d$ must lie in a fiber of $P_Z \to Z.$
It follows that 
$\sMbar_{0,3}(P_Z, d) \to $ is locally of the form $\sMbar_{0,3}(F, \bar d)\times U_i \to U_i$ for a suitable affine cover $\{U_i\}$ of $Z$, where $h: F \to P_Z$ is a fiber of $P_Z \to Z$ and $h_*(\widehat d) = d$.

We have the following commutative diagram
\[
\begin{tikzcd}
 \sMbar_{0,3}(P_Y, d) \ar[d, "\rm ev^Y"] \ar[dd, swap, bend right, "\rm ev^Y_i"] \ar[r] & \sMbar_{0,3}(P_Z, d) \ar[d, swap, "\rm ev^Z"] \ar[dd, bend left, "\rm ev^Z_i"] \\
     P_Y^3 \ar[d, "\pi_i^Y"] \ar[r, "\tilde \phi^3"] & P_Z^3 \ar[d, swap, "\pi_i^Z"] \\
    P_Y \ar[d] \ar[r, "\tilde \phi"] & P_Z \ar[d] \\
   Y \ar[r, "\phi"] & Z,
\end{tikzcd}
\]
in which all squares are Cartesian.  Here $P_Z^3$ (resp. $P_Y^3$) denotes the triple fiber product of $P_Z \to Z$ (resp. $P_Y \to Y$).

It is straightforward to check that the relative perfect obstruction theories of $\sMbar_{0,3}(P_Z, d) \to Z$ and $\sMbar_{0,3}(P_Y, d) \to Y$ are compatible with respect to $\phi$ in the sense of \cite{BF}.

By \cite[Proposition~5.10]{BF}, we observe that $$
(\tilde \phi^3)^! [\sMbar_{0,3}(P_Z, d)]^{vir}  = \phi^! [\sMbar_{0,3}(P_Z, d)]^{vir} = [\sMbar_{0,3}(P_Y, d)]^{vir}.$$

Let $\gamma = (\pi_1^Z)^* \alpha \cup (\pi_2^Z)^* \beta$.
We calculate
\begin{align*}
&\tilde \phi^* (\rm ev_3^Z)_* \left( (\rm ev_1^Z)^* \alpha \cup (\rm ev_2^Z)^* \beta
 \cap [\sMbar_{0,3}(P_Z, d)]^{vir}
\right) \\
=& \tilde \phi^* (\pi_3^Z)_* \left( \gamma \cap \rm ev^Z_* [\sMbar_{0,3}(P_Z, d)]^{vir} \right)
\\
=& (\pi_3^Y)_* (\tilde \phi^3)^*\left( \gamma \cap \rm ev^Z_* [\sMbar_{0,3}(P_Z, d)]^{vir} \right) \\
=& (\pi_3^Y)_* \left( (\tilde \phi^3)^* \gamma \cap (\tilde \phi^3)^* \rm ev^Z_* [\sMbar_{0,3}(P_Z, d)]^{vir} \right) \\
=& (\pi_3^Y)_* \left( (\tilde \phi^3)^* \gamma \cap  \rm ev^Y_* (\tilde \phi^3)^![\sMbar_{0,3}(P_Z, d)]^{vir} \right) \\
=& (\pi_3^Y)_* \left( (\tilde \phi^3)^* \gamma \cap  \rm ev^Y_* [\sMbar_{0,3}(P_Y, d)]^{vir} \right) \\
=& (\pi_3^Y)_* \rm ev^Y_* \left( (\rm ev^Y)^*((\tilde \phi^3)^* \gamma) \cap   [\sMbar_{0,3}(P_Y, d)]^{vir} \right) \\ 
=& (\rm ev_3^Y)_* \left( (\rm ev^Y_1)^* \tilde \phi^*\alpha \cup (\rm ev^Y_2)^* \tilde \phi^*\beta \cap   [\sMbar_{0,3}(P_Y, d)]^{vir} \right) 
\end{align*}

The first and last lines give the degree $d$ contribution to $\tilde \phi^* (\alpha \star \beta)$ and $\tilde \phi^* \alpha \star \tilde \phi^* \beta$ respectively.
\end{proof}

Theorem~\ref{t:homom} yields a  splitting principle for extremal quantum cohomology.

\begin{corollary}\label{c:qsp}
Let $p: P_Z \to Z$ be a fiber bundle as above. 
    Let $V \to Z$ be a vector bundle.  
    There exists a map of smooth varieties $\phi: Y \to Z$ such that the Chern roots of $\phi^*V$ are well-defined classes in $H^*(Y)$, and 
    $$\tilde \phi^*: QH^*_{ext}(P_Z) \to QH^*_{ext}(P_Y)$$
    is an injective ring homomorphism.
    \end{corollary}

\begin{proof}
  Let $\phi: Y \to Z$ denote the flag bundle of $V $, whose fiber over a point $z$ parametrizes complete flags in $V |z.$  
Then 
$$\phi^*: H^*(Z) \to H^*(Y)$$
is injective and $\phi^*V$ is filtered by a sequence of tautological sub-bundles.  The successive quotients of these sub-bundles are line bundles whose first Chern classes give the Chern roots of $\phi^*(V).$
By Theorem~\ref{t:homom}, $$\tilde \phi^*: QH^*_{ext}(P_Z) \to QH^*_{ext}(P_Y)$$
is a ring homomorphism.  It is injective by the Leray--Hirsch theorem.
\end{proof}

\subsection{Extremal quantum cohomology of the local model}

Next we focus on the special case of projective bundles.
Suppose $V \to Z$ is a vector bundle of rank $r$.
Let $\rho_1, \ldots, \rho_r$ denote the Chern roots of $V$.  
Let $H$ be the first Chern class of $\mathcal O_{\PP(V)}(1).$
As a ring,
\begin{equation}\label{e:Pcoh}
    H^*(\PP(V)) = H^*(Z)[H]\left/\br{ \prod_{i=1}^r (\rho_i + H) = 0}.\right.
\end{equation}

We compute the extremal quantum cohomology of $\PP(V)$.  More generally, let $V' \to Z$ be a vector bundle of rank $s$ with Chern roots $\sigma_1, \ldots, \sigma_s$.
Recall that $T$ is the total space of $$V'(-1):=\pi^*V' \otimes \mathcal O_{\PP(V)}(-1).$$
Then 
\begin{align}
c_1(T) = c_1(Z)  + \sum_{i=1}^r (\rho_i + H) + \sum_{j=1}^s (\sigma_j - H) 
= (r-s)H + c_1(Z)  + c_1(V) + c_1(V'), \label{first-chern-T}
\end{align}
where by abuse of notation, we denote $\pi^*(\gamma)$ simply by $\gamma$ for $\gamma \in H^*(Z)$.

Similar to Section~\ref{s:exc1}, let $[L]$ be the class of a line in a fiber of $\PP(V) \subset T \to Z.$  We specialize parameters as in \eqref{modified-pi-extremal2}, and set \begin{equation}\label{special}
       Q^d = 1 \text{ if } d = k[L], Q^d = 0 \text{ otherwise, and } \bt = \bt + \ln(q)c_1(T).
   \end{equation} 
\begin{theorem}\label{Jqlh}
Assume $r-s>1$.     
Let $\bt = \sum t^i \phi_i$ be a formal class pulled back from $H^*(Z)$.
With the specialization \eqref{special}, we have the following formulas for the $J$-functions. 
\begin{enumerate}
    \item \label{i1} 
    For $\PP(V)$, we have 
    \begin{equation}\label{J function}
        J^{\PP(V)}(q, \bt, z) = z e^{\bt/z}q^{c_1(\PP(V))/z} \sum_{d \geq 0}  q^{rd} \prod\limits_{i=1}^r \prod\limits_{m=1}^d \frac{1}{(\rho_i + H + mz)}.
    \end{equation}
    \item \label{i2} 
    For $T$, we have 
    \begin{equation}\label{J function2}
        J^T(q, \bt, z) = z e^{\bt/z}q^{c_1(T)/z} \sum_{d \geq 0} q^{(r-s)d}  \frac{\prod\limits_{j=1}^s \prod\limits_{m=0}^{d-1} (\sigma_j - H - mz)}{\prod\limits_{i=1}^r \prod\limits_{m=1}^d (\rho_i + H + mz)}.
    \end{equation}
\end{enumerate}
\end{theorem}
\begin{proof}
    In the split case, formula \eqref{J function}  is the $Q = 0$ specialization of   Novikov parameters of the $I$-function obtained by Brown in \cite[Corollary~1]{Brown}, and \eqref{J function2} then follows by applying the twisted theory of \cite{Coates-Givental}. 

    A more direct proof that applies also to the non-split setting may be obtained by following Bertram's argument in \cite{Ber}, which applies to the relative setting as well.  Indeed \eqref{J function} appears in Section~7 of \emph{loc. cit.}  The verification of \eqref{J function2} requires considering the obstruction bundle 
    $$e\left(R^1 \pi_{\cc C *} f^*(V'(-1) )\right).$$ 
    The generalization of Bertram's argument to this case is proven by Lee in \cite{Lee}.
\end{proof}

\begin{remark}\label{r333}
\begin{enumerate}
    \item The calculations of $I/J$-functions for projective bundles in more general settings have been studied extensively, e.g. in \cite{Ber, Ele, Brown, FL, IK}. 
    \item     We note that both the numerator and denominator appearing in \eqref{J function} are symmetric in the respective Chern roots, thus the expression is well-defined even when $V$ and $V'$ do not split.
    \item \label{i333} In the case $r -s \leq 1$, the right-hand side  \eqref{J function2} is not of the form $\one z + (\text{lower order terms in } z)$ and therefore cannot be equal to the $J$ function of $T$.  Nevertheless 
    it can be shown using \cite{Brown}, \cite{Coates-Givental}, and Lemma~\ref{l:r3} that \eqref{J function2}  lies on the Lagrangian cone of $T$, and therefore gives an $I$-function for $T.$  We will not need this fact in this paper.
\end{enumerate}
\end{remark}

Theorem \ref{Jqlh} implies the following:
\begin{theorem}
 We have the following presentations:
  \begin{align}\label{e:r1}
    QH^*_{ext}(\PP(V)) &=  H^*(Z)[H][[q]]\left/\br{ \prod_{i=1}^r (\rho_i + H)  =  q^r}.\right.
\end{align}
    and 
\begin{align}\label{e:r2}
    QH^*_{ext}(T) &=  H^*(Z)[H][[q]]\left/\br{ \prod_{i=1}^r (\rho_i + H)  = q^{r-s} \prod_{j=1}^s (\sigma_j -H)}.\right.
 \end{align}
\end{theorem}
\begin{proof}
By Corollary~\ref{c:qsp}, we can assume without loss of generality that the Chern roots of $V$ and $V'$ are well-defined cohomology classes in $H^*(Z).$

We will prove \eqref{e:r2}. Equation~\eqref{e:r1} follows from the same argument, or may be viewed as the special case $s=0$.
Note that the dimension of the right-hand side of \eqref{e:r2} is $$r \dim(H^*(Z)) = \dim (H^*(T))$$ over $\CC[[q]].$   It therefore suffices to show that 
$$\prod_{i=1}^r (\rho_i + H)  = q^{r-s} \prod_{j=1}^s (\sigma_j -H)\in QH^*_{ext}(T).$$

Specialize $\bt \in H^*(Z)$ in Theorem~\ref{Jqlh} to 
$$\bt  = \sum_{i=1}^{r} a^i \rho_i + \sum_{j=1}^s b^j \sigma_j + l c_1(Z).$$
Define the operator $\partial_{H}$ by
$$
\partial_{H}:= \frac{1}{r-s}\left(q \partial/\partial q - \sum_{i=1}^r  \partial/\partial a^i + \sum_{j=1}^s  \partial/\partial b^j -  \partial/\partial l\right).$$
Note that with this specialization, by \eqref{first-chern-T} we have
$$\partial_{H} e^{\bt/z}q^{c_1(T)/z} = \frac{H}{z} e^{\bt/z}q^{c_1(T)/z}.$$

Observe by direct computation that $J^T(q, \bt, z)$ satisfies the quantum differential equation
$$P(z \partial/\partial a^i, z\partial/\partial b^j, z\partial_H, q, z) J^T(q, \bt, z) = 0,$$
where 
$P(z \partial/\partial a^i, z\partial/\partial b^j, z\partial_H, q, z)$ is given by
$$
\prod_{i=1}^r \left(z \partial/\partial a^i + z\partial_H \right) - q^{r-s} \prod_{j=1}^s \left(z \partial/\partial b^j - z\partial_H \right).
$$
By \cite[Theorem~10.3.1]{CK} \cite[Lemma~2]{Pand}, this implies that
$P(\rho_i, \sigma_j, H, q, 0) = 0\in QH^*_{ext}(T).$
\end{proof}

\subsection{More on $J$-functions}

It will be useful to rewrite the (modified) $J$-functions defined in  \eqref{modified-pi-extremal2}, \eqref{modified-pi-extremal} using Gamma functions. 
Applying the calculation
$$z^{{\dim \over 2}}z^{\mu}(\rho_i+H+mz)=(\rho_i+H+m)z$$
to \eqref{J function}, 
we obtain 
\begin{equation*}\label{modified-J-projective-base}
    \wtJ^{\PP(V)}(q, \bt, -z)
    =e^{-\sum_{i}z^{\deg \phi_i -1}t^i\phi_i}\wtJ^{\PP(V)}(q, -z)
\end{equation*} 
where
\begin{equation}\label{modified-J-projective}
    \wtJ^{\PP(V)}(q, -z)
    =
    \sum_{d\geq 0}\left({q\over z}\right)^{rd-c_1(\PP(V))}\prod_{i=1}^{r}\frac{\Gamma(\rho_i + H -d ) }{\Gamma(\rho_i + H)}.
   \end{equation} 
By using the $J$-function formula in \eqref{J function2}, a similar calculation shows that the modified extremal $J$-function $\widetilde{J}^{T}(q, \bt, -z)$ of the local model $T$ is given by 
\begin{equation}\label{modified-J-local-base}
\widetilde{J}^{T}(q, \bt, -z)=e^{-\sum_{i}z^{\deg \phi_i -1}t^i\phi_i}\widetilde{J}^{T}(q, -z)
\end{equation}
where 
\begin{equation}
\label{pullback-J-gamma}
\wtJ^{T}(q, -z)=
\sum_{d \geq 0} \left({q\over z}\right)^{(r-s)d-c_1(T)}\prod_{i=1}^r \frac{\Gamma(\rho_i + H -d)}{\Gamma(\rho_i + H)}\prod_{j=1}^s \frac{(-1)^{d}\ \Gamma(1-\sigma_j + H)}{\Gamma(1-\sigma_j + H - d)}.
\end{equation}

For the local model $T'$, let $H':=c_1\left(\mathcal{O}_{\PP(V')}(1)\right)$. We consider the $I$-function:
\begin{equation}
        I^{T'}(q', \bt, z) =  e^{\bt/z} q'^{-c_1(T')/z}
        \sum_{d \geq 0} {q'}^{(r-s)d} \frac{\prod\limits_{i=1}^r\prod\limits_{m=0}^{d-1} (\rho_i - H' - mz)}{\prod\limits_{j=1}^s \prod\limits_{m=1}^d (\sigma_j + H' + mz)}.
    \end{equation}
\begin{remark}\label{r:T'cone}
    Although we will not need this fact in what follows, by Remark~\ref{r333} (\ref{i333}), $I^{T'}(q', \bt, z)$ lies on Givental's overruled Lagrangian cone for $T'$.
\end{remark}
Following \eqref{central-charge-modified-J}, for any $E\in K^0(T')$, we define an extremal $I$-function central charge of $E$ by 
\begin{equation}
\label{i-function-central-charge}
Z^{I^{T'}}(E)(q', \bt):=\left\langle \widetilde{I}^{T'}(q', \bt, -z), \widehat{\Gamma}_{T'}{\rm Ch}(E)\right\rangle^{T'}.
\end{equation}
Here $\widetilde{I}^{T'}(q', \bt, -z)$ is a modified extremal $I$-function given by
\begin{equation}
\label{I-function-T'}  
\widetilde{I}^{T'}(q', \bt, -z):=z^{c_1(T')}z^{\dim T'\over 2}z^{\mu} I^{T'}(q', \bt, -z)
    = e^{-\sum_{i}z^{\deg \phi_i -1}t^i\phi_i} 
    \widetilde{I}^{T'}(q', -z)
\end{equation}
where 
\begin{equation}
    \label{extremal-i-function}
      \widetilde{I}^{T'}(q', -z)=\sum_{d \geq 0} \left({q' z}\right)^{(r-s)d+c_1(T')}\prod_{j=1}^s \frac{\Gamma(\sigma_j + H' -d)}{\Gamma(\sigma_j + H')}\prod_{i=1}^r \frac{(-1)^{d}\ \Gamma(1-\rho_i + H')}{\Gamma(1-\rho_i + H' - d)}.
\end{equation}

\section{The quantum spectrum}\label{s:spectrum}
Let $T$ be as in the previous section. 
Let $R$ denote the ring
\begin{align}
     \label{e:qcoh}
       R=  H^*(Z)[H][[q]]\left/\br{ \prod_{i=1}^r (\rho_i + H)  = q^{r-s} \prod_{j=1}^s (\sigma_j -H)}.\right.
         \end{align}

\begin{prop}\label{p:eigen}
The operator of multiplication by  $c_1(T)$ in $R$ has an eigenvalue of 0 of multiplicity $s\cdot\dim(H^*(Z))$, and $r-s$ nonzero eigenvalues, each of multiplicity $\dim(H^*(Z))$, occurring at
$$\{\lambda_k(q):=(r-s)e^{-\pi\sqrt{-1}{2k+s\over r-s}} q \mid 0\leq k < r-s\}.$$
     \end{prop}
    
\begin{proof}
First note that by the explicit presentation of \eqref{e:qcoh} it is immediate that for $\alpha, \beta \in H^*(Z) \subset R$, the product $\alpha \cdot \beta$ in $R$ agrees with the cup product in $H^*(Z)$.  I.e., there is no quantum correction when multiplying two classes pulled back from the cohomology of the base.

    We will use the grading on $H^*(Z)$.  
    Let $S_0= H^*(Z)$ and let $R_0$ denote the ring 
    \eqref{e:qcoh}.
    Let $d$ be the largest integer such that $S_0^d \neq 0$ where $S_0^i$ denotes the $i$th graded piece of $S_0.$ 
    
    Choose a nonzero element $\alpha_0 \in S^d_0$.  The product of $\alpha_0$ with any homogeneous element of $S_0$ of positive degree is zero.
    Because 
    \begin{equation}\label{c1T}
        c_1(T) = (r-s)H + \gamma
    \end{equation}
    for some $\gamma \in H^{2}(Z)$,
        we observe that $c_1(T)^k \alpha_0$ is equal to $((r-s)H)^k \alpha_0.$
Let $W_0 \subset R_0$ denote the span of $H^k \alpha_0$ for $0 \leq k \leq r-1$.  Then $c_1(T)$ preserves $W_0$.  With respect to the basis $\{((r-s)H)^k \alpha_0\}_{k=0}^{r-1}$, the matrix for this operator on $W_0$ is 
$$M_{c_1(T)} = \left(
\begin{array}{ccccccc}
0& 0&  &\cdots && 0& 0 \\
1& 0 &  &\cdots && 0& 0 \\
0& 1 &  &\cdots && 0& 0 \\
\vdots &\vdots&&&&\vdots& \vdots \\
0 & 0 &  &\ddots&& 0 & (r-s)^{r-s} (-1)^s q^{r-s}
\\
\vdots &\vdots&&&&\vdots& \vdots \\
\\
0& 0 &  &\cdots && 0& 0 \\
0& 0 &  &\cdots && 1& 0 
\end{array}
\right)$$
where the $s+1$st row of $M_{c_1(T)}$ is  $(0, \ldots, 0, 1, 0, \ldots, 0, (r-s)^{r-s} (-1)^s q).$ 
    
The characteristic polynomial of this matrix is
\begin{align}\label{e:cp}
   \lambda^s(\lambda^{r-s} - (r-s)^{r-s} (-1)^s q^{r-s}),
\end{align} 
thus the operator of multiplacation by $c_1(T)$ on $W_0$ has an eigenvalue of $0$ with multiplicity $s$ and $r-s$ nonzero eigenvalues, each of multiplicity 1, occuring at 
$(r-s)e^{-\pi\sqrt{-1}{2k+s\over r-s}} q$
for $0\leq k< r-s$.

We proceed by induction.  Let $S_i$ denote the graded ring $S_{i-1}/<\alpha_{i-1}>$.  
Let 
$$R_i = R_{i-1}/W_{i-1} = S_i[H][[q]]\left/\br{ \prod_{i=1}^r (\rho_i + H)  =q^{r-s} \prod_{j=1}^s (\sigma_j -H)}.\right.$$
Use the grading of $S_i$ to choose $\alpha_i \in S_i$ such that the product of $\alpha_i$ with any homogeneous element of $S_i$ of positive degree is zero.  Let $W_i \subset R_i$ denote the span of $H^k \alpha_i$ for $0 \leq k \leq r-1$.  The characteristic polynomial of multiplication by $c_1(T)$ is again \eqref{e:cp}.  This process repeats $\dim(H^*(Z))$ times before terminating.
\end{proof}

To prove Theorem~\ref{qspec} for general $X$, we will require a slight generalization of the Proposition~\ref{p:eigen}.
Let $\bar R$ be a subspace of $R$ which contains 
$$\prod_{j=1}^s (\sigma_j -H)\alpha = 
e(V'(-1))\alpha$$
for all $\alpha \in R$
and is preserved by multiplication by $c_1(X)$.
\begin{prop}\label{p:qm2}
The operator of multiplication by $c_1(X)$ on $\bar R$ has 
    r-s nonzero eigenvalues, each of multiplicity $\dim(H^*(Z))$, occurring at    
     $$\{\lambda_k(q):=(r-s)e^{-\pi\sqrt{-1}{2k+s\over r-s}} q \mid 0\leq k < r-s\}.$$
     All other eigenvalues are 0.
\end{prop}

\begin{proof}
     The proof follows from a slight modification of the previous proof. 
Let $S_0 = H^*(Z)$ as in the  proof of Proposition~\ref{p:eigen} and choose $\alpha_0 \in S_0$ as before.  
Define $\bar W_0$ to be the subspace spanned by 
$e(V'(-1)) \alpha_0 H^k$ for $0 \leq k\leq r-s-1.$
Because $\alpha_0$ is of top degree in $S_0,$ $c_1(T)^k e(V'(-1))\alpha_0 = ((r-s)H)^k e(V'(-1))\alpha_0$, and 
\begin{align*}
    c_1(T)^{r-s} e(V'(-1)) \alpha_0 &= (r-s)^{r-s} H^r \alpha_0 \\
    &= (r-s)^{r-s} e(V(1)) \alpha_0 \\
    &=(r-s)^{r-s}(-1)^sq^{r-s} e(V'(-1)) \alpha_0.
\end{align*} 
  Thus $\bar W_0$ is preserved by multiplication by $c_1(T)$.  Following the previous proof, the characteristic polynomial of this operator on $W_0$ is seen to be
$$( \lambda^{r-s} - (r-s)^{r-s} (-1)^s q^{r-s}).$$

At step $i$ define $S_i = S_{i-1}/<\alpha_{i-1}>.$  Define $\bar R_i$ to be the quotient of $\bar R_{i-1}$ by 
$\bar W_{i-1}$.  
Choose $\alpha_i \in S_i$ of top degree.  By assumption we have that  
$e(V'(-1)) \alpha_i$ lies in  $\bar R_i$.  Define $\bar W_i$ to be the subspace spanned by 
$e(V'(-1)) \alpha_i H^k$ for $0 \leq k\leq r-s-1,$ and proceed inductively.

The process terminates after $\dim(H^*(Z))$ steps.  We conclude that the characteristic polynomial of $c_1(T)$
acting on $\bar R$ 
contains $$( \lambda^{r-s} - (r-s)^{r-s} (-1)^s q^{r-s})^{\dim(H^*(Z))}$$
as a factor.  All other eigenvalues must be zero by the previous proposition. \end{proof}

\begin{corollary}\label{coro-c1-multiplication}
    Theorem~\ref{qspec} holds.
\end{corollary}
\begin{proof}
Recall that by Theorem~\ref{t:func}, $$i^*: QH^*_{ext}(X) \to QH^*_{ext}(T)$$ is a ring homomorphism.  Furthermore, by definition of $T = \tot(N_{F|X})$, the tangent bundle is preserved under $i^*:$ 
$$i^*(T(X)) = T(F)\oplus N_{F|X} = T(T).$$
Consequently, $$i^*(c_1(X) \star_q \alpha) = c_1(T) \star_q i^*\alpha.$$

By the projection formula, the image of $$i^*: QH^*_{ext}(X) \to QH^*_{ext}(T)$$
contains $e(\pi^*V' \otimes \mathcal O_{\PP(V)}(-1)) =i^* \circ j_*(1).$  Therefore by Proposition~\ref{p:qm2}, 
we have $r-s$ nonzero eigenvalues, each of multiplicity $\dim(H^*(Z))$, occurring at
     $(r-s)e^{-\pi\sqrt{-1}{2k+s\over r-s}} q$
     for $0\leq k < r-s$, and all other eigenvalues in $\text{im}(i^*)$ are zero.  On the other hand, 
     in the kernel of $i^*$ there are no quantum corrections, so $c_1(X) \star_q - = c_1(X) \cup - .$  Consequently, the restriction of the operator $c_1(X) \star_q - $ to $\ker(i^*)$ is nilpotent, and the thus the remaining eignevalues  must all be zero.  By a dimension argument, we conclude that the eigenvalue of zero occurs with multiplicity $\dim H^*(X')$.
\end{proof}

Following the same argument as in  Proposition~\ref{p:eigen}, we can determine the eigenvalues of $E(\bt) \star_{q, \bt} - $ in the local model $T$ for more general $\bt$.
\begin{corollary}\label{c:Estar}
    For $\bt$ pulled back from  $H^{\geq 2}(Z)$, the operator 
    $$E(\bt) \star_{q, \bt} - $$
    on $H^*(T)$ has the same eigenvalues as $c_1(T) \star_q - $.
\end{corollary}
\begin{proof}
Because we are restricting to the extremal quantum cohomology, we assume $d = k[L]$ with $L$ a line contracted by $\pi$. By  Proposition~\ref{fiber-moduli}, the composition 
$$\pi \circ \rm ev_i: \sMbar_{g,n}(T, d) \to T \to Z$$
is independent of the marked point $i$ used to evaluate.

Consequently, for $\bt$ pulled back from  $H^*(Z)$, 
\begin{align*}
    \br{\alpha, \beta, \gamma, \bt, \ldots, \bt}_{0,n+3, d}^X &= \br{\alpha \cdot \bt^n, \beta, \gamma, 1, \ldots, 1}_{0,n+3, d}^X. 
\end{align*}
By the fundamental class axiom, the right-hand side is zero whenever $n>0$.  Therefore, for $\bt$ pulled back from  $H^*(Z)$,
$$\alpha \star_{q, \bt} \beta = \alpha \star_q \beta.$$  It therefore suffices to compute the eigenvalues of $E(\bt) \star_q - .$

As in \eqref{c1T}, 
$E(\bt) = (r-s)H + \gamma$
    with $\gamma$ an element of $H^{\geq 2}(Z)$.
The proof of Proposition~\ref{p:eigen} then goes through identically after replacing $c_1(T)$ by $E(\bt)$.
\end{proof}

\section{Asymptotics of projective bundles: the split case}\label{s:projbund}

Suppose $V \to Z$ is a vector bundle of rank $r$.
Assume in this section that $V$ splits as $V = \oplus_{i=1}^r L_i$.
Consider the trivial action of 
 $\TT = (\CC^*)^{r}$ on $Z$ and view $V$ as $\TT$-equivariant vector bundle on $Z$ by letting the $i$th factor of $\TT$ act on $L_i$ by scaling. 
 By abuse of notation, denote the \emph{equivariant} first Chern classes of $L_1, \ldots, L_r$ by $\rho_1, \ldots, \rho_r$.
Then as a vector space we have $$H^*_\TT(Z) = H^*(Z)[\rho_1, \ldots, \rho_r] = H^*(Z) \otimes_\CC H^*_\TT(pt).$$
With the induced action of $\TT$ on $\PP(V)$ we have
\begin{equation}\label{e:eqcoh}
     H^*_\TT(\PP(V)) = H^*(Z)[\rho_1, \ldots, \rho_r, H]\left/\br{ \prod_{i=1}^r (\rho_i + H)  =  0}.\right.
\end{equation}
We will denote by $R_\TT$ the ring $H^*_\TT(pt)$, and by $S_\TT$ its localization with respect to non-zero homogeneous elements.  We will also make use of the completion
$\hat S_\TT$, defined by
\[
\hat S_\TT = \left\{ \sum_{d \in \ZZ} a_d \mid \text{there exists $d_0 \in \ZZ$ such that $a_d = 0$ for $d < d_0$}\right\},
\]
where $a_d$ lies in  the degree $d$ homogeneous part of $S_\TT.$

\subsection{A residue formula}
Let $\pi: \PP(V) \to Z$ be the (equivariant) projection map and denote by
$$\pi_*: H^*_\TT(\PP(V)) \to H^*_\TT(Z)$$ 
the pushforward in cohomology.  We use the injectivity of $\pi^*$ and  \eqref{e:eqcoh} to identify $H^*_\TT(Z)$ with a subspace of $H^*_\TT(\PP(V))$. I.e., given a class $\alpha \in H^*_\TT(Z)$, the pullback $\pi^*(\alpha)$ will sometimes be denoted simply by $\alpha$ when no confusion is likely to 
result.

\begin{lemma}\label{l:res}
Let 
$$f(H, \rho_1, \ldots, \rho_r) \in H^*_\TT(\PP(V))\otimes_{H^*_\TT(pt)} \hat S_\TT.$$  
Then we have the following equality
    \begin{equation} \label{e:residues}
    \pi_*\left(f(H, \rho_1, \ldots, \rho_r)\right) = \sum_{j=1}^r
\operatorname{Res}_{H= - \rho_j}\left( \frac{
f(H, \rho_1, \ldots, \rho_r)}{\prod_{i=1}^r 
    (\rho_i +H)}
    \right).
\end{equation}
\end{lemma}

\begin{proof}
The proof follows from  the Atiyah--Bott localization theorem applied to $H^*_\TT(\PP(V))$.
Denote by $F_j$  the fixed locus  $\PP(L_j) \subset \PP(V)$.  On this locus, $\cc O_{\PP(V)}(-1)$ restricts to $L_j$, thus we have the restriction 
$$H|_{F_j} = -\rho_j.$$

The normal bundle is given by 
$$e(N_{F_j|\PP(V)}) = \prod_{i \neq j}\left(\rho_i+  H \right)|_{F_j} 
= \prod_{i \neq j}\left(\rho_i- \rho_j \right),$$
the Atiyah--Bott localization theorem states that 
\begin{align*} \pi_* \left(f( H, \rho_1, \ldots, \rho_r)\right) = &\sum_{j=1}^r
\frac{f( H, \rho_1, \ldots, \rho_r)|_{F_j}}{e(N_{F_j|\PP(V)})} \\
= &\sum_{j=1}^r
\frac{f(-\rho_j, \rho_1, \ldots, \rho_r)}{\prod_{i\neq j} 
    (\rho_i -\rho_j)} \\
= &
\sum_{j=1}^r
\operatorname{Res}_{ H= - \rho_j}\left( \frac{
f( H, \rho_1, \ldots, \rho_r)}{\prod_{i=1}^r 
    (\rho_i + H)}
    \right).
\end{align*}

\end{proof}

Viewing $\rho_i$ as a small complex number, for $d$ an integer the function $\Gamma(x + \rho_i - d)$ is regular at $x = 0$ and thus has a well-defined power series expansion.  
This allows us to define the $H^*(\PP(V))$-valued function 
$\Gamma(H + \rho_i - d)$.

Lemma~\ref{l:res} therefore gives, in particular, the following equality for all $d>0$,
\begin{equation} \label{e:residues3}
    \pi_*\left(
    \prod_{i=1}^r (\rho_i +H) \Gamma(\rho_i +H - d)
    \right) = 
    \sum_{j=1}^{r}
    \operatorname{Res}_{H= -\rho_{j}} \prod_{i=1}^r \Gamma(\rho_i +H - d).
\end{equation}

\subsection{Weak asymptotic classes}
Now we use the residue formula above to study the asymptotic behavior of the flat sections 
$$\left\{\Phi^{\PP(V)}_\TT(q, z)\left(\widehat{\Gamma}_{\PP(V)}{\rm Ch}(\mathcal{O}_{\PP(V)}(m))\pi^*(\alpha)\right)\mid \alpha\in H^*_\TT(Z), m\in\mathbb{Z}\right\}$$
when $V$ is a splitting bundle.
According to \eqref{leading-term}, when specializing to $\bt=0$ and pairing the flat sections above with the flat identity $\one$, we obtain 
\begin{equation}
    \label{leading-term-proj-bundle}
\left\langle\widetilde{J}^{\PP(V)}_\TT(q, -z), \widehat{\Gamma}_{\PP(V)}{\rm Ch}(\mathcal{O}_{\PP(V)}(m))\pi^*(\alpha)\right\rangle^{\PP(V)}_\TT.
\end{equation}
In this section, we only consider the asymptotic behavior of such terms and show weak asymptotic classes of Definition \ref{weak-asymp-def}.

Following \cite[Section 1]{Meijer} \cite[Section 5.2]{Luke}, we 
consider the Meijer $G$-function 
\begin{eqnarray*}
G_{0,r}^{r,0}\left(
\begin{array}{l}
\emptyset\\
\lambda_1, \ldots, \lambda_r
\end{array}\bigg\vert \ t \right)
= \frac{1}{2 \pi \sqrt{-1}} \int_L t^{x} \prod_{i=1}^r \Gamma(\lambda_i - x) dx 
\end{eqnarray*}
where $\lambda_1, \ldots, \lambda_r$ are fixed small complex numbers and $L$ is a path from $-\sqrt{-1}\infty$ to $\sqrt{-1}\infty$ such that all the poles of $\Gamma(\lambda_i - x)$ lie to the right of the path. 
 This function may be expanded as a power series in $\lambda_1, \ldots, \lambda_r$ with coefficients analytic functions of $t$.  
This allows us to define the $H^*_\TT(Z)$-valued function
\begin{equation}
G_{0,r}^{r,0}\left(
\begin{array}{l}
\emptyset\\
\rho_1, \ldots, \rho_r
\end{array}\bigg\vert \ t \right)
= \frac{1}{2 \pi \sqrt{-1}} \int_L t^{x} \prod_{i=1}^r \Gamma(\rho_i - x) dx.   
\end{equation}
Near $t=0$, it has a residue expression 
\begin{align} \label{expansion-meijer-bundle}
  G_{0,r}^{r,0}\left(
\begin{array}{l}
\emptyset\\
\rho_1, \ldots, \rho_r
\end{array}\bigg\vert \ t \right)
=&  \sum_{d \geq 0} \sum_{j=1}^{r}
    \operatorname{Res}_{x= \rho_j + d} t^{x}\prod_{i=1}^r \Gamma(\rho_i -x)\\ \label{e:psires1}
    =&  \sum_{d \geq 0} \sum_{j=1}^{r}
    \operatorname{Res}_{H= -\rho_j} t^{d-H}\prod_{i=1}^r \Gamma(\rho_i +H - d). 
\end{align}

\begin{prop}\nonumber
For fixed $q\neq 0$, near $z=\infty$, 
the term \eqref{leading-term-proj-bundle} is given by 
$$\int_Z \widehat{\Gamma}_{Z}\alpha 
\left({z\over q}\right)^{c_1(Z)+c_1(V)} 
 G_{0,r}^{r,0}\left(
\begin{array}{l}
\emptyset\\ \label{weak-projective-bundle}
\rho_1, \ldots, \rho_r
\end{array}\bigg\vert \ e^{-2\pi\sqrt{-1}m}\left({q\over z}\right)^r \right).$$
\end{prop}
\begin{proof}
Recall that the Gamma class $\widehat{\Gamma}_{\PP(V)}$ is defined by formula \eqref{gamma-class-vect}. 
We will also consider the relative Gamma class
$\widehat \Gamma_\pi := \widehat \Gamma_{T\pi}$
of the relative tangent bundle $T\pi$, given by 
$\widehat \Gamma_\pi = \widehat \Gamma_{\PP(V)}/ \pi^*(\widehat \Gamma_Z).$
Using the relative Euler sequence for $\PP(V)$ 
$$0 \to \mathcal O_Z \to \pi^*V(1) \to T\pi \to 0,$$
we compute 
\begin{equation}
\label{relative-gamma-proj-bundle}
\widehat \Gamma_\pi = \widehat \Gamma_{\pi^*V(1)} = \prod_{i=1}^r \Gamma( 1 + \rho_i + H)\quad \in H^*(\PP(V)).
\end{equation}
Using \eqref{e:residues3}, we compute
   \begin{eqnarray*}
        &&\left\langle\widetilde{J}^{\PP(V)}_{\TT}(q, -z), \widehat{\Gamma}_{\PP(V)}{\rm Ch}(\mathcal{O}_{\PP(V)}(m))\pi^*(\alpha)\right\rangle_{\TT}^{\PP(V)}\\
        &=&\left\langle\widetilde{J}^{\PP(V)}_{\TT}(q, -z), 
        (-1)^{2mH}\widehat{\Gamma}_{\pi} \pi^*(\widehat{\Gamma}_{Z} \alpha)\right\rangle_{\TT}^{\PP(V)}\\
        &=&\int_{\PP(V)}\pi^*(\widehat{\Gamma}_{Z} \alpha)
    \sum_{d\geq 0}\left({q\over z}\right)^{r(d-H)-c_1(Z)-c_1(V)}(-1)^{2mH} \prod_{i=1}^{r}(\rho_i + H)\Gamma(\rho_i + H -d ) \\
        &=&\int_Z \left(\widehat{\Gamma}_{Z} \alpha\right)
\pi_*\left(\sum_{d\geq 0}\left({q\over z}\right)^{r(d-H)-c_1(Z)-c_1(V)}
(-1)^{2mH}  \prod_{i=1}^{r}(\rho_i + H)\Gamma(\rho_i + H -d )\right)\\
&=&\int_Z \left(\widehat{\Gamma}_{Z} \alpha\right)
\sum_{d\geq0} \sum_{j=1}^{r}  \operatorname{Res}_{H= -\rho_{j}}
\left({q\over z}\right)^{r(d-H)-c_1(Z)-c_1(V)}
(-1)^{2mH}
\prod_{i=1}^r \Gamma(\rho_i +H - d).
    \end{eqnarray*}
Here the first equality is based on 
$${\rm Ch}(\mathcal{O}_{\PP(V)}(m))=e^{2\pi\sqrt{-1}m H}=(-1)^{2mH}.$$
The second equality uses the calculation \eqref{modified-J-projective} of $\wtJ^{\PP(V)}$  together with \eqref{relative-gamma-proj-bundle} and 
$$c_1(\PP(V))=rH+c_1(Z)+c_1(V).$$
The third equality uses the projection formula for $\pi: \PP(V) \to Z$ and the last is \eqref{e:residues3}.
Substituting  $$t=e^{-2\pi\sqrt{-1}m}\left({q\over z}\right)^r,$$ and applying \eqref{expansion-meijer-bundle}, we obtain the result.
\end{proof}

From here we apply Barnes' asymptotic expansion formula (Theorem \ref{exponential-asymptotic} in Appendix \ref{appendix-asymptotic}) to the Meijer $G$-function in Proposition \ref{weak-projective-bundle} to obtain 
\begin{prop}
For $r>1$, the element in \eqref{leading-term-proj-bundle} 
admits asymptotic expansion
$$\exp\left({-{re^{{-2\pi\sqrt{-1}m}/r}q\over z}}\right)
\int_Z \left(\widehat{\Gamma}_{Z} \alpha\right)  \sum_{n=0}^{\infty}c_n \left({z\over q}\right)^{n+{r-1\over 2}+c_1(Z)}$$
as $z\to 0$ with $$\left|\arg\left({z\over q}\right)+{2m\pi\over r}\right|< {(r+1)\pi\over r},$$ 
where the coefficients $c_n\in H^*(Z)$ are determined recursively.
\end{prop}
In particular, according to Definition \ref{weak-asymp-def}, this implies 
\begin{prop}
\label{main-thm-bundle}
Let $r>1$.
For any $\alpha\in H^*(Z)$ and $m\in \mathbb{Z}$, the element $\widehat{\Gamma}_{\PP(V)}{\rm Ch}(\mathcal{O}_{\PP(V)}(m))\pi^*(\alpha)$ is a weak asymptotic class with respect to $re^{{-2\pi\sqrt{-1}m}/r}q$ 
and $\arg(z/q)=-2m\pi/r$.
\end{prop}

\section{Asymptotics for the non-zero eigenvalues: the local model}\label{s:nzflip}
In this section, we generalize Proposition \ref{main-thm-bundle} from projective bundle $\PP(V)$ to the (split) local model for  standard flips.

\subsection{Reduction to local models}
We recall diagram \eqref{e:diagram} of standard flips.  Denote by $p$ the map $p: X \to \bar X$ and let $j: F \to X$ be the inclusion of the exceptional locus $p^{-1}(Z)$.
Define $$\Psi_m^X(\alpha)=\left(2\pi\sqrt{-1}\right)^s j_*\left({\rm Ch}( \mathcal{O}_F(m)){\rm Td}(N_{F|X})^{-1} \pi^*(\alpha)\right).$$

We want to show that the elements 
\begin{equation}
    \label{m-th-cohomology}
    \{\widehat{\Gamma}_X \Psi_m^X(\alpha)\vert \alpha\in H^*(Z)\} \subset H^*(X)
\end{equation}
 are weak asymptotic classes with respect to the nonzero eigenvalues. 
By Definition \ref{weak-asymp-def} and the formula \eqref{leading-term}, we need to calculate
\begin{equation}\label{interal-on-X}
    \left\langle\wtJ^X(q, -z),\widehat{\Gamma}_X \Psi_m^X(\alpha)\right\rangle^X.
\end{equation}
\begin{lemma}

For any $E\in K^0(Z),$ we have 
\begin{equation}\label{phi-m-todd}
{\rm Ch}(\Phi_m^X(E))=
\Psi_m^X\left(  {\rm Ch}(E)\right).
\end{equation}
\end{lemma}
\begin{proof}
Using the Grothendieck-Riemann-Roch formula for $j:F\to X$, the projection formula, and the short exact sequence \begin{equation}
\label{ses-TX}
0\to TF\to j^*(TX)\to N_{F|X}\to 0,
\end{equation} 
 we have 
$${\rm Ch}(j_*E)=\left(2 \pi \sqrt{-1}\right)^s j_*\left({\rm Ch}(E){\rm Td}(N_{F|X})^{-1}\right),$$
where here the factor of $\left(2 \pi \sqrt{-1}\right)^s$ is to account for the fact that we are using the modified Chern character and Todd class \eqref{modified-todd}.
Then the equation \eqref{phi-m-todd} follows from $\Phi_m^X(E)=j_*(\pi^*E\otimes \mathcal{O}_F(m)).$
\end{proof}

By the projection formula we have:
\begin{lemma}\label{Z-integral-weak}
For any $\alpha\in H^*(Z)$, the expression $\left\langle\wtJ^X(q, -z),\widehat{\Gamma}_X \Psi_m^X(\alpha)\right\rangle^X
$ is given by    
$$\left(2 \pi \sqrt{-1}\right)^s\int_Z  \pi_*\left(j^*\left(\wtJ^X(q, -z)\widehat{\Gamma}_X\right){\rm Ch}(\mathcal{O}_F(m)){\rm Td}(N_{F|X})^{-1}\right)\alpha.$$
In particular, for any $E\in K(Z)$, the central charge $Z(\Phi_m^X(E))$ is
$$\left(2 \pi \sqrt{-1}\right)^s\int_Z  \pi_*\left(j^*\left(\wtJ^X(q, -z)\widehat{\Gamma}_X\right){\rm Ch}(\mathcal{O}_F(m)){\rm Td}(N_{F|X})^{-1}\right){\rm Ch}(E).$$
\end{lemma}

The modified extremal $J$-function $\widetilde{J}^{T}(q, \bt, -z)$ of the local model $T$ is given by \eqref{pullback-J-gamma}. 
By a similar argument to Theorem~\ref{t:func}, we have
\begin{align}
j^*\left(\wtJ^X(q, -z)\right)\label{pull-back-of-X} &=\wtJ^T(q, -z)\\&=
\sum_{d \geq 0} \left({q\over z}\right)^{(r-s)d-c_1(T)} \prod_{i=1}^r \frac{\Gamma(\rho_i + H -d)}{\Gamma(\rho_i + H)}\prod_{j=1}^s \frac{(-1)^{d}\ \Gamma(1-\sigma_j + H)}{\Gamma(1-\sigma_j + H - d)}.\nonumber
\end{align}
From the short exact sequence \eqref{ses-TX} and \eqref{relative-gamma-proj-bundle}, we compute that 
\begin{equation}
\label{pullback-gamma-X}
j^*\widehat{\Gamma}_X=
\pi^*(\widehat{\Gamma}_Z)
\prod_{i=1}^r \Gamma(1 + \rho_i + H) \prod_{j=1}^s \Gamma(1 + \sigma_j - H).
\end{equation}
Applying the Euler reflection formula \eqref{euler-reflection} to the modified Todd class \eqref{modified-todd} for $N_{F|X}$,
we have 
   \begin{equation}\label{todd-normal} 
   \operatorname{Td}(N_{F|X})
   =\prod_{j=1}^s \Gamma(1+\sigma_j - H) \Gamma(1-\sigma_j + H) e^{\pi  \sqrt{-1} (\sigma_j - H)}.   
   \end{equation}
   
Denote $e^{\pi\sqrt{-1}}$ by $-1$ and define
$$\kappa_m:=\pi_*\left(\sum_{d \geq 0} \left({q\over z}\right)^{(r-s)(d-H)}(-1)^{2mH}\prod_{i=1}^r (\rho_i + H)\Gamma(\rho_i + H -d ) \prod_{j=1}^s \frac{(-1)^{(-\sigma_j + H-d)} }{\Gamma(1-\sigma_j + H - d)}\right).$$
Combining formulas \eqref{pull-back-of-X}, \eqref{pullback-gamma-X}, \eqref{todd-normal}, and
${\rm Ch}(\mathcal{O}_F(m))=(-1)^{2mH},$ 
we obtain 
\begin{prop}
We have 
\begin{equation}\label{pushforward-via-kappa}
\left\langle\wtJ^X(q, -z),\widehat{\Gamma}_X \Psi_m^X(\alpha)\right\rangle^X
=\left(2 \pi \sqrt{-1}\right)^s \int_Z  \left(\widehat{\Gamma}_Z\alpha\right) \left({z\over q}\right)^{c_1(Z)+c_1(V')+c_1(V)} \kappa_m.
\end{equation}
\end{prop}

\subsection{Local models of splitting bundles}
For the remainder of the section we restrict our attention to the ``split local model.''
Assume  that $V$ and $V'$ split as $V = \oplus_{i=1}^r L_i$ and $V' = \oplus_{j=1}^s M_j$. 
Let $\TT= (\CC^*)^{r+s}$ act trivially on $Z$, and lift this to an action on $V\oplus V'$ by letting $i$th factor of $\TT$ scale the $i$th factor of $V\oplus V'$. This induces an action of $\TT$
 on $T$ and  $T'$.  By abuse of notation, we will also use $\rho_i, \sigma_j$ to denote the $\TT$-equivariant Chern classes of $L_i, M_j$. 
As in the previous section, let $S_\TT$ be the localization of $H^*_\TT(pt)$ and $\hat S_\TT$ the completion.

In the equivariant setting, $\kappa_m$ may be expressed as a cohomology-valued Meijer $G$-function near $z=\infty$.
\begin{lemma}\label{meijer-G-exponent}
We have
$$\kappa_m 
=e^{-\pi\sqrt{-1}c_1(V')}G_{s,r}^{r,0}\left(
\begin{array}{c}
1-\sigma_1, \ldots, 1-\sigma_s\\
\rho_1,\ldots, \rho_r
\end{array}
\bigg\vert \ e^{-\pi\sqrt{-1}(2m+s)}\left({q\over z}\right)^{r-s}
\right).$$
\end{lemma}
\begin{proof}
Applying Lemma~\ref{l:res} yields
\begin{align*}
\kappa_m
=&\sum_{d \geq 0}\sum_{a=1}^{r}
    \operatorname{Res}_{H= -\rho_{a}}
    \left({q\over z}\right)^{(r-s)(d-H)}
    (-1)^{2mH}    \prod_{i=1}^r \Gamma(\rho_i + H -d )    \prod_{j=1}^s \frac{(-1)^{(-\sigma_j + H-d)} }{\Gamma(1-\sigma_j + H - d)}  
    \\
    =&\sum_{d \geq 0} \sum_{a=1}^{r}
    \operatorname{Res}_{x= \rho_{a}+d}
     \left({q\over z}\right)^{(r-s)x}
     (-1)^{2m(d-x)}
     \prod_{i=1}^r \Gamma(\rho_i - x) \prod_{j=1}^s \frac{(-1)^{(-\sigma_j - x)} }{\Gamma(1-\sigma_j - x)}  \\
    =&\frac{e^{-\pi\sqrt{-1}c_1(V')}}{2\pi \sqrt{-1}} \int_L
      {\prod\limits_{i=1}^r \Gamma(\rho_i - x)\over \prod\limits_{j=1}^s \Gamma(1-\sigma_j -x)}      \left({q\over z}\right)^{(r-s)x}
     e^{\pi\sqrt{-1}(-2mx-sx)} dx.
\end{align*}
    Here the path $L$ goes from $-\sqrt{-1}\infty$ to $\sqrt{-1}\infty$ so that all poles of $\Gamma(\rho_i-x)$ lie to the right of the path. Now the result follows from the definition of the Meijer $G$-function \eqref{meijer-def} in Appendix \ref{appendix-asymptotic}.
    \end{proof}

Recall from formula \eqref{evalue-m-th} in Theorem \ref{qspec} that the nonzero eigenvalues of $c_1(X)\star_q -$ are
$$\{\lambda_k(q)=(r-s)e^{-\pi\sqrt{-1}{2k+s\over r-s}} q\mid 0\leq k<r-s\}.$$
For $r-s>1$, we apply Barnes' asymptotic formula in \cite{Meijer, Luke} (see Theorem \ref{exponential-asymptotic}) to the Meijer $G$-function in Lemma \ref{meijer-G-exponent} and obtain an asymptotic expansion
\begin{eqnarray*}
&&G_{s,r}^{r,0}\left(
\begin{array}{c}
1-\sigma_1, \ldots, 1-\sigma_s\\
\rho_1,\ldots, \rho_r
\end{array}
\bigg\vert \ e^{-\pi\sqrt{-1}(2m+s)}\left({q\over z}\right)^{r-s}
\right)\\
&\sim& 
{(2\pi)^{r-s-1\over 2}\over (r-s)^{1\over 2}}\exp\left({-{\lambda_m(q)\over z}}\right) 
e^{-\pi\sqrt{-1}((2m+s)\theta)}\left({q\over z}\right)^{\theta}
\sum_{n=0}^{\infty}c_n \left({z\over q}\right)^n
\end{eqnarray*}
as $z\to 0$ with 
\begin{equation}
\label{arg-z-expoenential}
\left|\arg\left({z\over q}\right)+{\pi(2m+s)\over r-s}\right|<\left(1+{1\over r-s}\right)\pi.
\end{equation}
Here the coefficients $c_n\in H^*(Z)$ are determined recursively with $c_0=1$ and 
$$\theta={1-r-s\over 2(r-s)}+{c_1(V')+c_1(V)\over r-s}.$$
Using this asymptotic expansion formula and formula \eqref{pushforward-via-kappa}, we obtain 
\begin{prop}\label{exponential-asymptotic-local-model}
For any $\alpha\in H^*_\TT(Z)$,
$\left\langle\wtJ^T_\TT(q, -z),\widehat{\Gamma}_T \Psi_m^T(\alpha)\right\rangle^T_\TT$ 
admits the asymptotic expansion 
$$\left(2 \pi \sqrt{-1}\right)^s{(2\pi)^{r-s-1\over 2}\over (r-s)^{1\over 2}}e^{-{\lambda_m(q)\over z}} 
\int_Z  \left(\widehat{\Gamma}_Z \alpha\right)  \sum_{n=0}^{\infty}c_{n, m} \left({z\over q}\right)^{n+{r+s-1\over 2}-c_1(Z)}$$
as $z\to 0$ with $\arg(z)$ satisfying the condition \eqref{arg-z-expoenential} and $$c_{n,m}:=(-1)^{-(c_1(V')+(2m+s)\theta)}c_n.$$

\end{prop}

Using \eqref{modified-J-local-base}, one can repeat the computations above replacing $\wtJ^{T}_\TT(q, -z)$  with $\wtJ^{T}_\TT(q, \bt, -z)$ for $\bt$ a class pulled back from $H^{\geq 2}(Z) \subset H^{\geq 2}_\TT(Z)$, to obtain
\begin{corollary}
For any $\alpha\in H^*_\TT(Z)$,
$\left\langle\wtJ^T_\TT(q, \bt, -z),\widehat{\Gamma}_T \Psi_m^T(\alpha)\right\rangle^T_\TT$
admits the asymptotic expansion 
$$\left(2 \pi \sqrt{-1}\right)^s{(2\pi)^{r-s-1\over 2}\over (r-s)^{1\over 2}}e^{-{\lambda_m(q)\over z}} 
\int_Z   e^{-\sum_{i}z^{\deg \phi_i -1}t^i\phi_i}
\left(\widehat{\Gamma}_Z \alpha\right)\sum_{n=0}^{\infty}c_{n,m} \left({z\over q}\right)^{n+{r+s-1\over 2}-c_1(Z)}$$
as $z\to 0$ with $\arg(z)$ satisfying the condition \eqref{arg-z-expoenential}.
\end{corollary}

\begin{remark}
The action of $\TT$ on $Z$ is trivial.  Therefore, since $\bt$ is assumed to be pulled back from $H^{\geq 2}(Z)$, the class $\bt$ is nilpotent, as is  $c_1(Z)$.  In particular, the expression $e^{-\sum_{i}z^{\deg \phi_i -1}t^i\phi_i}$ is polynomial in $z$ and the term  $z^{c_1(Z)}$ in the asymptotic formula above is  a polynomial in $\ln(z)$. These terms therefore do not effect the exponential asymptotics of the expression.
\end{remark}

Applying the result in particular to $\alpha = {\rm Ch}(E)$ for $E\in K^0(Z)$ we obtain an asymptotic expansion for the central charge $Z(\Phi_m^T(E))(q, \bt)$.
We conclude
\begin{prop}\label{nzasymp2}
For any $\bt$ pulled back from $H^{\geq 2}(Z)$,
the elements
$$\widehat{\Gamma}_T\Psi_m^T(\alpha)=\widehat{\Gamma}_T \left(2\pi\sqrt{-1}\right)^s j_*\left({\rm Ch}( \mathcal{O}_F(m)){\rm Td}(N_{F|T})^{-1} \pi^*(\alpha)\right)
$$ 
are weak asymptotic classes with respect to $\lambda_m(q)$ and $\arg(z/q)={-2m-s\over r-s}\pi$.  

In particular, for any $E\in K_\TT^0(Z)$, $\widehat{\Gamma}_T {\rm Ch}(\Phi_m^T(E))$ is such a weak asymptotic class.

\end{prop}
\begin{remark}
    The fact that the  exponential asymptotics of $\widehat{\Gamma}_T\Psi_m^T(\alpha)$ are unchanged as we deform from $\bt= 0$ to $\bt$ pulled back from $H^{\geq 2}(Z)$ is
    consistent with Corollary ~\ref{c:Estar}, which verifies that  for such $\bt$, the eigenvalues of $E(\bt) \star_{q, \bt} - $ agree with those of $c_1(T)\star_q - $.
\end{remark}

\section{Asymptotics for the zero-eigenvalues: local model}\label{s:zflip}
Let $T \dasharrow T'$ be the split local model of the previous section, together with the previously defined action of $\TT$.
\subsection{Fourier Mukai transformation on equivariant cohomology}
\label{s-fourier-mukai}
  Consider the resolution 
\[
\begin{tikzcd}
        &\widehat T \ar[dl, swap, "\tau_T"] \ar[dr, "\tau_{T'} "] &\\
    T \ar[dr, swap, "\pi \circ k "]&& T' \ar[dl, "\pi ' \circ k'"] \\
    &Z&.
\end{tikzcd}
\]
 In the case of the  local model, we may also express $\widehat T$ as a GIT quotient.  Let $(\CC^*)^2$ act on 
$\tot( V \oplus \cc O_Z \oplus V')$ with weights 
\[
\left( 
\begin{array}{ccc}
    1&-1 &0  \\
     0& 1 & -1
\end{array}
\right).
\]
Then $\widehat T$ is the GIT quotient of this action with character given by $(1, -1)^\top.$  In homogeneous coordinates, the map $\tau_{T'}$ is given by $(v, l, v') \mapsto (vl, v')$ and $\tau_T$ is given by $(v, l, v') \mapsto (v, lv').$

Define $$\FM^T=(\tau_T)_* \circ (\tau_{T'})^*: K^0_\TT(T') \to K^0_\TT(T).$$

We define a linear map 
\begin{equation}
\label{FM-cohomology}
\UU^T: H^*_\TT(T')\otimes \hat S_\TT \to H^*_\TT(T)\otimes \hat S_\TT
\end{equation}
by
\begin{equation}
\label{FM-map}
\UU^T\left(\frac{\ii_j'}{e_\TT(N_{F_j'|T'})}\right)=\sum_{i=1}^r C_{ij} \frac{\ii_i}{e_\TT(N_{F_i|T})},
\end{equation}
where $F_j '$ is the $j$th fixed locus (isomorphic to $Z$) in $T'$, and $\ii_j'$ is the fundamental class on $F_j'$ (similarly for $F_i$ and $\ii_i$), and  
\begin{align}
\label{FM-coefficients}
C_{ij} :=& e^{-(s-1)\pi\sqrt{-1}(\rho_i+\sigma_j)} \prod_{l \neq j}\frac{e^{\pi\sqrt{-1}(\rho_i+\sigma_l)} - e^{-\pi\sqrt{-1}(\rho_i+\sigma_l)}}{e^{\pi\sqrt{-1}(-\sigma_j+\sigma_l)} - e^{\pi\sqrt{-1}(\sigma_j - \sigma_l)}}\\
=& e^{-(s-1)\pi\sqrt{-1}(\rho_i+\sigma_j)} \prod_{l \neq j}
\frac{\sin {\pi(\rho_i+\sigma_l)}}
{\sin {\pi(-\sigma_j+\sigma_l)}}.\nonumber
\end{align}

Using localization and the method in \cite{PSW}, one can compute that 
 \begin{prop}
 The following diagram commutes
 \begin{equation}\label{d:U}
     \begin{tikzcd}
         K^0_\TT(T') \ar[r, "\FM^T"] \ar[d, "{\rm Ch}"] &  K^0_\TT(T) \ar[d,  "{\rm Ch}"] \\
         H^*_\TT(T')\otimes \hat S_\TT \ar[r, "\UU^T"] & H^*_\TT(T)\otimes \hat S_\TT,
     \end{tikzcd}
 \end{equation}
where the vertical maps are the (equivariant) modified Chern character. 
\end{prop}
\begin{proof}

Let $V = \oplus_{i=1}^r L_i$ and $V' = \oplus_{j=1}^s M_j$.  By abuse of notation, we use the same notation for the pullback of these line bundles to $T$ and $T'$.

For $1 \leq j \leq s$, let  $e_j \in K^0_\TT(T')$ denote the class $$\prod_{l \neq j} ( \cc O_{T'} \ominus {k'}^* {\pi'}^*M_l^\vee \otimes \cc O_{T'}(1)).$$  This class restricts to the invertible class $\wedge^\bullet N_{F_j'|\PP(V')}^\vee$ on $F_j'$, where $F_j'$ is the $j$th fixed locus of the $\TT$ action on $T'$.  It restricts to zero on $F'_l$ for $l \neq j$.  By the same computation as \cite[Proposition~2.6]{PSW}, one computes
\begin{equation}\label{FMe}
 \FM^T(e_j) = \prod_{l \neq j}( \cc O_{T} \ominus {k}^* {\pi}^* M_l^\vee \otimes \cc O_{T}(1)), 
\end{equation}
where here (again abusing notation slightly) $M_l$ now denotes pulled back line bundle to $T$. 

From \eqref{FMe} it follows that for $1 \leq i \leq r$, 
$${\rm Ch}\left(\iota_i^* (\FM^T (e_j))\right) 
= \prod_{l \neq j} \left(1 - e^{-2\pi\sqrt{-1}(\rho_i +\sigma_l)}\right),$$ 
where $\iota_i: F_i \to T$ is the inclusion of the $i$th fixed locus.
On the other hand, 
$${\rm Ch} ((\iota'_j)^* e_j) = \prod_{l \neq j} 
\left(1 - e^{2\pi\sqrt{-1}(\sigma_j - \sigma_l)}\right)$$
where $\iota'_j: F'_j\to T'$ is the inclusion of the $j$th fixed locus. 

The ratio of the above two expressions is
\begin{align*}
     \prod_{l \neq j} \frac{1 - e^{-2\pi\sqrt{-1}(\rho_i +\sigma_l)}}{1 - e^{2\pi\sqrt{-1}(\sigma_j - \sigma_l)}}, 
\end{align*}
which simplifies to $C_{ij}$ defined in \eqref{FM-coefficients}.
The commutativity of \eqref{d:U} then follows from the localization theorem, after observing that a basis for  $K^0_\TT(T')$ is given by 
$$\{e_j \otimes \alpha\}_{1 \leq j \leq s, \alpha}$$ where $\alpha$ ranges over a basis of $K^0_\TT(Z),$ pulled back to $T'$.
\end{proof}

\subsection{Weak tame asymptotic classes}

We first recall from \cite{Meijer, Luke} an asymptotic formula for the Meijer $G$-function
\begin{equation}\label{tame-Meijer-G}
  G_{s,r}^{r,1}\left(t || a_l\right):=G_{s,r}^{r,1}\left(
\begin{array}{c}
a_l, a_1, \ldots, a_{l-1}, a_{l+1}, \cdots, a_s\\
b_1, \ldots, b_r
\end{array}
\bigg\vert \ t\right).  
\end{equation}
For the purpose of the asymptotic formula, we need the assumption that 
\begin{equation}
    \label{assumption-tame}
    b_i-b_k\notin \mathbb{Z} \quad \forall \quad i\neq k, \quad a_l-b_i\notin\mathbb{Z}_{+}, \quad a_l-a_j\notin \mathbb{Z}.
\end{equation}

In this section, we only need to consider a special case when $a_l=1$.  
Recall that $\rho_i$, $\sigma_j$ are the equivariant Chern classes of $L_i$ and $M_j$ respectively as in Section \ref{s-fourier-mukai}.
Set
\begin{equation}\label{pi-function}
a_j=1-\sigma_j+\sigma_l; \quad
b_i=\rho_i+\sigma_l.
\end{equation} 
Viewing the equivariant parameters as general small complex numbers, we can check easily that the assumption in \eqref{assumption-tame} holds for the choice.

Following \cite[Section 5.2]{Luke}, we replace the path $L$ in $G_{s,r}^{r,1}\left(t || a_l\right)$ by a path beginning and ending at $+\infty$ and encircling all poles of $\Gamma(b_i-x)$, $i=1, \ldots, r$, once in the negative direction, but none of the poles of $\Gamma(x)$. Because $s<r$, the replacement does not change the result of the integral and thus according to \cite[Section 5.2 (7)]{Luke}, for $|t|<1$, we obtain a series 
\begin{eqnarray}
&&G_{s,r}^{r,1}\left(t || a_l=1\right)\nonumber\\
&=&\sum_{k=1}^{r}{\Gamma(b_k)
\prod\limits_{i\neq k}\Gamma(b_i-b_k)\over \prod\limits_{j\neq l}\Gamma(a_j-b_k)}
\, t^{b_k} \,  _sF_{r-1}\left(
\begin{array}{l}
1+b_k-a_1, \ldots, 1+b_k-a_{s}\\
1+b_k-b_1, \ldots, 1+b_k-b_{r}
\end{array}; (-1)^{s-r-1}t\right)\nonumber\\
&=&\sum_{k=1}^{r}{\pi\over\sin(\pi b_k)}\prod_{i\neq k}{\pi\over\sin(\pi (b_i-b_k))} 
\prod\limits_{i=1}^{r}{1\over \Gamma(1+b_k-b_i)}
\prod\limits_{j=1}^{s}{1\over \Gamma(a_j-b_k)} \nonumber\\
&&\times t^{b_k} \sum_{d=0}^{\infty}((-1)^{s-r-1}t)^d\ 
{\prod\limits_{j=1}^{s}\prod\limits_{m=0}^{d-1}(1+b_k-a_j+m)\over \prod\limits_{i=1}^{r}\prod\limits_{m=1}^{d}(b_k-b_i+m)}.\nonumber
\end{eqnarray}
Here we use the Euler reflection formula \eqref{euler-reflection}  to obtain the second equality. 
Recall the definition of Gamma classes in \eqref{gamma-class-minus}, it is  easy to see 
\begin{equation}
\label{gamma-minus-local}
    \widehat\Gamma_{T,-}\Big\vert_{H=-\rho_k}=\widehat\Gamma_{Z,-}\prod\limits_{i=1}^{r}\Gamma(1+b_k-b_i)\prod\limits_{j=1}^{s}\Gamma(a_j-b_k).
\end{equation}
Applying the formula \eqref{pullback-J-gamma}, we obtain
\begin{equation}
   \label{J-function-localized}
    \widetilde{J}^T_\TT(q,-z)\Big\vert_{H=-\rho_k}=\sum_{d \geq 0} \left({q\over z}\right)^{(r-s)(d+\rho_k)-c_1(Z)-c_1(V')-c_1(V)}
{\prod\limits_{j=1}^{s}\prod\limits_{m=0}^{d-1}(1+b_k-a_j+m)\over \prod\limits_{i=1}^{r}\prod\limits_{m=1}^{d}(-b_k+b_i-m)}.
\end{equation}
Now we set 
\begin{equation}
\label{change-coordinate-tame}
    t=e^{\pi\sqrt{-1}(1-s)}\left({q\over z}\right)^{r-s},
\end{equation}
the calculations above implies the following equality near $t=0$ (or $z=\infty$)
\begin{eqnarray*}
&&
\sum_{k=1}^{r}
{e^{\pi\sqrt{-1}(1-s)b_k}\pi^r\over \sin(\pi b_k)\prod\limits_{i\neq k}\sin(\pi (b_i-b_k))}
{\widetilde{J}^T_\TT(q,-z)\over \widehat\Gamma_{T,-}}
\Big\vert_{H=-\rho_k}\\
&=&
\left({z\over q}\right)^{(r-s)\sigma_l+c_1(Z)+c_1(V')+c_1(V)}
{G_{s,r}^{r,1}\left(t || a_l=1\right)\over \widehat\Gamma_{Z,-}}.
\end{eqnarray*}

On the other hand, we evaluate the integral along a path beginning and ending at $-\infty$ and encircling all poles of $\Gamma(x)$ 
once in the positive direction, but none of the poles of $\Gamma(b_i-x)$, $i=1, \ldots, r$. 
According to \cite[Thereom A]{Meijer} (see Theorem \ref{algebraic-asymptotic} in Appendix \ref{appendix-asymptotic}), 
the Meijer $G$-function $G_{s,r}^{r,1}\left(t || a_l=1\right)$ admits for large values of  $|t|$ with 
$|\arg(t)|<{r-s\over 2}\pi+\pi$ the asymptotic expansion
\begin{eqnarray}
&&G_{s,r}^{r,1}\left(t || a_l=1\right)\label{tame-asymptotic-formula}\\
&\sim&
\sum_{d=0}^{\infty} {(-t)^{-d}\prod\limits_{i=1}^{r}\Gamma(b_i+d)\over \prod\limits_{j=l}^{s}\Gamma(a_j+d)}\nonumber\\
&=&\prod\limits_{i=1}^{r}{\pi\over \sin(\pi b_i)} 
\prod\limits_{i=1}^{r}{1\over \Gamma(1-b_i)}
\prod\limits_{j=1}^{s}{1\over \Gamma(a_j)}
\sum_{d=0}^{\infty}(-t)^{-d}
{\prod\limits_{i=1}^{r}\prod\limits_{m=0}^{d-1}(b_i+m)\over 
\prod\limits_{j=1}^{s}\prod\limits_{m=1}^{d}(a_j-1+m)}.\nonumber
\end{eqnarray}
Again, we use the Euler reflection formula \eqref{euler-reflection} in the equality. Recall the formula \eqref{gamma-class-minus},
we obtain 
\begin{equation}
\widehat\Gamma_{T',-}\Big\vert_{H'=-\sigma_l}
=\widehat\Gamma_{Z,-}\prod\limits_{i=1}^{r}\Gamma(1-b_i)\prod\limits_{j=1}^{s}\Gamma(a_j).
\end{equation}
A similar calculation to \eqref{J-function-localized} 
with the $I$-function formula \eqref{extremal-i-function} for $T'$ gives 
$$\widetilde{I}^{T'}_\TT(q',-z)\Big\vert_{H'=-\sigma_l}=\sum_{d \geq 0} \left(q'z\right)^{(r-s)(d+\sigma_l)+c_1(Z)+c_1(V')+c_1(V)}
{\prod\limits_{i=1}^{r}\prod\limits_{m=0}^{d-1}(b_i+m)\over 
\prod\limits_{j=1}^{s}\prod\limits_{m=1}^{d}(-a_j+1-m)}.
$$

Now the asymptotic formula \eqref{tame-asymptotic-formula} implies
\begin{prop}
\label{mellin-barnes}
There exists an asymptotic expansion 
\begin{eqnarray*}
\sum_{k=1}^{r}{e^{\pi\sqrt{-1}(1-s)b_k}\pi^r\over \sin(\pi b_k)\prod\limits_{i\neq k}\sin(\pi (b_i-b_k))}{\widetilde{J}^T_\TT(q,-z)\over \widehat\Gamma_{T,-}}\Big\vert_{H=-\rho_k}
\sim\prod\limits_{i=1}^{r}{\pi\over  \sin(\pi b_i)} {\widetilde{I}^{T'}_\TT\left(q^{-1},-z\right)\over \widehat\Gamma_{T',-}}\Big\vert_{H'=-\sigma_l}
\end{eqnarray*}
as $z\to 0$ with 
\begin{equation}\label{argument-tame}
\left|\arg\left({z\over q}\right)-{(1-s)\pi\over r-s}\right|<{\pi\over 2}+{\pi\over r-s}.
\end{equation}
\end{prop}

We use this asymptotic formula to prove the main result of the section. 
\begin{prop}\label{weak-tame}
For any $\alpha\in H^*_\TT(T')\otimes S_\TT$, the cohomology class $\widehat{\Gamma}_T \UU^T(\alpha)$ is a weak tame asymptotic class with $\arg(z/q)={(1-s)\over r-s}\pi.$
\end{prop}
\begin{proof}
We calculate the asymptotic behavior of 
$$\left\langle\widetilde{J}^T_\TT(q, -z),\widehat{\Gamma}_T \UU^T \alpha\right\rangle^T_\TT, \quad \text{for all } \alpha=\frac{\ii_l'}{e_\TT(N_{F_l'|T'})}, \quad l=1, \cdots, s.$$
Using the definition of $\UU^T$ in \eqref{FM-coefficients},  the Todd class formula in \eqref{todd-from-gamma}, the localization calculations 
$$
\begin{dcases}
e_\TT(N_{F_k|T})=\prod_{j=1}^{s}(\sigma_j+\rho_k)\prod_{i\neq k}(\rho_i-\rho_k)
=\prod_{j=1}^{s}(1+b_k-a_j)\prod_{i\neq k}(b_i-b_k),\\
e_\TT(N_{F_l'|T'})
=\prod_{i=1}^{r}(\rho_i+\sigma_l)\prod_{j\neq l}(\sigma_j-\sigma_l)
=\prod_{i=1}^{r}b_i\prod_{j\neq l}(1-a_j),
\end{dcases}
$$
and the asymptotic formula in Proposition \ref{mellin-barnes}, we obtain
\begin{eqnarray*}
&&\left\langle\widetilde{J}^T_\TT(q, -z),\widehat{\Gamma}_T \UU^T \left(\frac{\ii_l'}{e_\TT(N_{F_l'|T'})}\right)\right\rangle^T_\TT\\
&=&\sum_{k=1}^{r}\left\langle\widetilde{J}^T_\TT(q, -z), { {\rm Td}(T) \over e^{c_1(T)}\widehat{\Gamma}_{T,-}} \frac{C_{kl}}{e_\TT(N_{F_k|T})} \ii_k\right\rangle^T_\TT\\
&=&\int_Z { {\rm Td}(Z) \over e^{c_1(Z)}}
\prod_{j\neq l}{\pi\over\sin(\pi(1-a_j))}\sum_{k=1}^{r}
{e^{\pi\sqrt{-1}(1-s)b_k}\pi^r\over \sin(\pi b_k)\prod\limits_{i\neq k}\sin(\pi (b_i-b_k))}
{\widetilde{J}^T_\TT(q,-z)\over \widehat\Gamma_{T,-}}
\Big\vert_{H=-\rho_k}\\
&\sim&\int_Z { {\rm Td}(Z) \over e^{c_1(Z)}}
\prod_{j\neq l}{\pi\over\sin(\pi(1-a_j))}
\prod\limits_{i=1}^{r}{\pi\over \sin(\pi b_i)}
{\widetilde{I}^{T'}_\TT\left(q^{-1},-z\right)\over \widehat\Gamma_{T',-}}
\Big\vert_{H'=-\sigma_l}\\
&=&
\left\langle\widetilde{I}^{T'}_\TT\left(q^{-1}, -z\right),
{ {\rm Td}(T') \over e^{c_1(T')}\widehat{\Gamma}_{T',-} } \frac{\ii_{l}'}{e_\TT(N_{F'_{l}|T'})} \right\rangle^{T'}_\TT\\
&=&\left\langle\widetilde{I}^{T'}_\TT\left(q^{-1}, -z\right),\widehat{\Gamma}_{T'}  \frac{\ii_{l}'}{e_\TT(N_{F'_{l}|T'})} \right\rangle^{T'}_\TT.
\end{eqnarray*}
Because $\{\frac{\ii_l'}{e_\TT(N_{F_l'|T'})}\}$ forms a basis of $H^*_\TT(T')\otimes S_\TT$,
this implies that for any $\alpha\in H^*_\TT(T')\otimes S_\TT$, 
\begin{equation} \label{e:FM and central charge}
\left\langle\widetilde{J}^T_\TT(q, -z),\widehat{\Gamma}_T \UU^T \alpha\right\rangle^T_\TT \sim 
\left\langle\widetilde{I}^{T'}_\TT\left(q^{-1}, -z\right),\widehat{\Gamma}_{T'}  \alpha \right\rangle^{T'}_\TT.\end{equation}
Now the result is a consequence of \eqref{leading-term}. 
\end{proof}

As in the previous section, if we replace $\widetilde{J}^T_\TT(q, -z)$ with $\widetilde{J}^T_\TT(q, \bt, -z)$ for
$\bt$ pulled back from $H^{\geq 2}(Z)$, we obtain an analogous formula to \eqref{e:FM and central charge}:
\begin{equation} \label{e:FM and central charge2}
\left\langle\widetilde{J}^T_\TT(q, \bt, -z),\widehat{\Gamma}_T \UU^T \alpha\right\rangle^T_\TT \sim  
\left\langle\widetilde{I}^{T'}_\TT\left(q^{-1},\bt,  -z\right),\widehat{\Gamma}_{T'}  \alpha \right\rangle^{T'}_\TT.\end{equation}
Since the expression $e^{-\sum_{i}z^{\deg \phi_i -1}t^i\phi_i}$ is polynomial in $z$.  This gives the following.
\begin{corollary}
    \label{weak-tame2}
For $\bt$ pulled back from $H^{\geq 2}(Z)$ and for any $\alpha\in H^*_\TT(T')\otimes S_\TT$, 
$\widehat{\Gamma}_T \UU^T(\alpha)$ is a weak tame asymptotic class with $\arg(z/q)={(1-s)\over r-s}\pi.$
\end{corollary}

From the diagram \eqref{d:U}, we have ${\rm Ch}(\FM^T(E'))=\UU^T {\rm Ch}(E')$. 
So the asymptotic expansion is compatible with the 
Fourier-Mukai transformation in Section \ref{s-fourier-mukai}. 
\begin{prop}
\label{algebraic-asymptotic-local}
For any $E'\in K^0_\TT(T')$ and $\bt$ pulled back from $H^{\geq 2}(Z)$, the central charge $Z(\FM^T(E'))$ admits the asymptotic expansion
$$Z(\FM^T(E'))(q, \bt)\sim 
Z^{I^{T'}}(E')\left(q^{-1}, \bt\right)
$$
as $z\to 0$ with $\arg(z)$ satisfies the condition \eqref{argument-tame}.
Here $Z^{I^{T'}}(E')$ is defined in 
\eqref{i-function-central-charge}.
\end{prop}

\begin{remark}
    Proposition~\ref{algebraic-asymptotic-local} may be viewed as a wall crossing result for the discrepant transformation $T\dasharrow T'$.  For such local models, we expect that the same method as used above, when applied to Brown's $I$-functions for toric bundles \cite{Brown}, should yield an analogous result \emph{without} restricting to extremal curve classes.  This is related to Iritani's work on discrepant toric wall crossings \cite{Iri-toric}.  However we do not pursue this here, as it is peripheral to the main goals of this paper.  
\end{remark}

\section{Reductions}\label{s:reductions}

In the previous two sections we found weak asymptotic classes for the split local model.  In this section we use those computations together with a series of reductions to determine strong asymptotic classes in all cases and prove Theorem~\ref{main-theorem}.

\subsection{Weak implies strong in the local model}\label{ss:ws}
We show that in the  local model the notion of strong asymptotic classes and weak asymptotic classes are equivalent.

Using the presentation from above, we have
\begin{align}\label{QHT2}
    QH^*_{ext}(T) &=  H^*(Z)[H][[q]]\left/\br{ \prod_{i=1}^r (\rho_i + H)  = q^{r-s} \prod_{j=1}^s (\sigma_j -H)}.\right.
\end{align}

Let us now assume that $V$ and $V'$ split as $V = \oplus_{i=1}^r L_i$ and $V' = \oplus_{j=1}^s M_j$.  Then there is an action of 
 $\TT = (\CC^*)^{r+s}$ on $T, T'$ induced by scaling each of the summands of $V \oplus V'.$  By abuse of notation, we will also use $\rho_i, \sigma_j$ to denote the equivariant Chern classes of $L_i, M_j$.  With this notation, then $QH^*_{ext}(T)_\TT$ still has a presentation as $H^*_\TT(Z)[H][[q]]$ modulo the relation in \eqref{QHT2}.

\begin{prop}\label{wtos}
Let $\beta$ be pulled back from $H^*_\TT(Z)$, for $i \geq 0$, the expression $\br{c_1(T)^i\beta, \Phi^T_\TT(q, \bt, z) \alpha}_\TT^T$ is obtained from $\br{ 1, \Phi^T_\TT(q, \bt, z) \alpha}_\TT^T$ by a sequence of derivatives. 
\end{prop}
\begin{proof}
Let $\nabla^\TT$ denote the equivariant Dubrovin connection, and $S^T_\TT(q, \bt, z)$  the equivariant solution matrix, where we take $\bt$ to be a general class pulled back from $H^*_\TT(Z).$  Let $\br{ , }^T_\TT$ denote the equivariant Poincar\'e pairing, which is defined via localization.
 
Then, for  $\beta$ pulled back from $H^*_\TT(Z)$, by \eqref{e:dubcon}, we have
\begin{align*}
        \br{\beta, \Phi^T_\TT(q, \bt, z) \alpha}^T_\TT & =  z\br{\nabla_\beta^\TT 1, \Phi^T_\TT(q, \bt, z) \alpha}^T_\TT \\ \nonumber
        & = z\partial/\partial_\beta \br{ 1, \Phi^T_\TT(q, \bt, z) \alpha}^T_\TT - z\br{ 1, \nabla_\beta^\TT \Phi^T_\TT(q, \bt, z) \alpha}^T_\TT \\ \nonumber
        & = z\partial/\partial_\beta \br{ 1, \Phi^T_\TT(q, \bt, z) \alpha}^T_\TT.
\end{align*}
The Dubrovin connection for $q$ is given by
$$\nabla_{q \frac{\partial}{\partial q}}^\TT = q \frac{\partial}{\partial q} + \frac{1}{z}c_1(T) \star - .$$ Using the presentation \eqref{QHT2}, for $i<r$, 
\begin{align}\label{e:derivs}
        \br{c_1(T)^i\beta, \Phi^T_\TT(q, \bt, z) \alpha}^T_\TT & =  z^i\left(q\partial/\partial_q\right)^{i} \br{\beta , \Phi^T_\TT(q, \bt, z) \alpha}^T_\TT \\
        & =  z^{i+1}\left(q\partial/\partial_q\right)^{i} \partial/\partial_\beta \br{ 1, \Phi^T_\TT(q, \bt, z) \alpha}^T_\TT.\nonumber
        \end{align}
\end{proof}
Recall that $\lambda_m(q)=(r-s)e^{-\pi\sqrt{-1}{2m+s\over r-s}} q$. We aslo notice that the set $$\{c_1(T)^i\beta |\; i \geq 0, \beta \text{ pulled back from } H^*_\TT(Z)\}$$ spans $H^*_\TT(T)$. 
We immediately conclude the following.
\begin{prop}
  \label{p:stac}
For $\alpha \in H^*_\TT(Z)$, the element $\widehat{\Gamma}_T\Psi_m^T(\alpha)$
is a strong asymptotic class with respect to $\lambda_m(q)$
and $\arg(z/q)={-2m-s\over r-s}\pi.$
   
For  $\alpha\in H^*_\TT(T')$, the element $\widehat{\Gamma}_T \UU^T(\alpha)$ is a strong tame asymptotic class with $\arg(z/q)={1-s\over r-s}\pi.$
\end{prop}
\begin{proof}
The first statement is  Proposition~\ref{wtos} and Corollary~\ref{nzasymp2}.
    The second statement is  Proposition~\ref{wtos} and Corollary~\ref{weak-tame2}.
\end{proof}

Consider the subtorus 
\begin{align*}
\ST = &(\CC^*)^2 \hookrightarrow \TT = (\CC^*)^{r+s} \\
&	(s_1, s_2) \mapsto (s_1, \ldots, s_1, s_2, \ldots, s_2).
\end{align*}
where the first factor of $\CC^*$ in $\ST$  embeds diagonally in $(\CC^*)^r$ and the second factor embeds diagonally in $(\CC^*)^s$.  With this embedding, the induced action of $\ST$ on $V \oplus V'$ is given by
$$(s_1, s_2) \cdot (v, v') = (s_1 v, s_2 v').$$

One can take a ``partial nonequivariant limit'' of Proposition~\ref{p:stac} to obtain the same result in $\ST$-equivariant cohomology:
\begin{corollary}
  \label{p:stac2}
Let $T$ be a local model such that $V$ and $V'$ split.  For $\alpha \in H^*_\ST(Z)$, 
the element
$\widehat{\Gamma}_T\Psi_m^T(\alpha)$
is strong asymptotic class with respect to $\lambda_m(q)$
and $\arg(z/q)={-2m-s\over r-s}\pi.$
   
For  $\alpha\in H^*_\ST(T')$, the element $\widehat{\Gamma}_T \UU^T(\alpha)$ is a strong tame asymptotic class 
with $\arg(z/q)={1-s\over r-s}\pi.$
\end{corollary}

\subsection{Reducing to the split case}\label{redsplit}
In this section and what follows we assume that $\bt= 0 $.  
Suppose $T \dasharrow T'$ is a local model  such that $V, V' \to Z$ do not necessarily split.  In this section we show that the computation of strong asymptotic classes for $T$ can be reduced to the computation for a related split local model.

We first prove a general deformation statement.
    Suppose \begin{equation}\label{ses1}
    0 \to A \to V \to Q \to 0
    \end{equation} is an exact sequence of vector bundles over $Z$.  Let $T, T' \to Z$ be as above, and let $ T_{sp}, T'_{sp}$ be defined similarly, but replacing $V$ with $A \oplus Q$.  That is, consider the projective bundle $\pi_{sp}: \PP(A \oplus Q) \to Z$ and define the total spaces 
    $$T_{sp} = \tot \left(\pi_{sp}^* V' \otimes \cc O_{\PP(A \oplus Q)}(-1)\right),$$
    $$T_{sp}' = \tot \left( {\pi '}^* (A\oplus Q) \otimes \cc O_{\PP(V')}(-1)\right).$$  
    
    Let $\ST$ also act on $A \oplus Q$ by scaling by the first factor (so that \eqref{ses1} is $\ST$-equivariant).
    \begin{lemma}\label{l:ctosplit}
There are canonical isomorphisms
$$H^*_\ST(T) = H^*_\ST( T_{sp}), \hspace{.5 cm} H^*_\ST(T') = H^*_\ST( T'_{sp}).$$  
  Under these isomorphisms, we have the identification
  $$\Phi^T_\ST(q, z) = \Phi^{T_{sp}}_\ST(q, z).$$
  
  Furthermore, under this identification, 
  $\UU^T = \UU^{T_{sp}}.$
  \end{lemma}
  \begin{proof}
  The first statement is immediate from the $\ST$-equivariant generalization of \eqref{e:Pcoh}.  By \eqref{J function2} the $J$-functions of $T$ and $T_{sp}$ agree under this isomorphism of cohomology rings.  The $J$-functions determine the operators $\Phi^T$ and $\Phi^{T_{sp}}$ by \eqref{leading-term} and \eqref{e:derivs}.  The second claim follows.
  
  To prove the final claim we consider a family of local models specializing from $T$ to $T_{sp}$.
     Let $d: V \to Q$ denote the quotient map.   Over $Z \times \AF^1,$ define a map of vector bundles 
\begin{align*}
    g: A \oplus Q &\to Q \\ (a, q) &\mapsto d(a) - t\cdot q,
\end{align*}
where $\AF^1 = \operatorname{Spec}(\CC[t])$ and, by abuse of notation we denote also by $A, Q$ the pullbacks of these vector bundles to $Z \times \AF^1.$  We have, 
over $t=1,$ $\ker(g)_1 = V$, while over $t=0$, 
 $$\ker(g)_0 = A \oplus Q.$$

Let $\tilde  \pi:\PP(\ker(g)) \to Z \times \AF^1 $ be the projective bundle and define
 $$U = \tot\left(\tilde \pi^* V'\otimes \cc O_{\PP(\ker(g))}(-1)\right).$$  Then $T = U_1$ and $\bar T = U_0$, where  $U_t$ denotes the fiber over $t \in \AF^1$.  
Define $\widehat U$ and $U'$ the same as in the definition of $\widehat T$ and $T'$, but replacing $V \to Z$ with $\ker(g) \to Z \times \AF^1$ in each case.  We have the following diagram.
\[
    \begin{tikzcd}
      & T_{sp}'=U'_0 \ar[r, "i'_0"] & U' & T'=U'_1 \ar[l, swap, "i'_1"]\\
      & \widehat{T}_{sp} = \widehat{U}_0   \ar[r, "\widehat i_0"] \ar[d, swap, "\tau_0"] \ar[u, "\tau'_0"]& \widehat{U} \ar[d, swap, "\widehat \tau"] \ar[u, "\widehat {\tau'}"]& \widehat T = \widehat{U}_1 \ar[l, swap, "\widehat i_1"] \ar[d, swap, "\tau_1"] \ar[u, "\tau'_1"]\\
      & T_{sp}=U_0 \ar[r, swap, "i_0"] & U  & T=U_1 \ar[l, "i_1"]    \end{tikzcd}
    \]
With this notation, $$\UU^T = {\tau_1}_*\left( \operatorname{Td}(T\tau_1) {\tau_1'}^*(-) \right),  \UU^{T_{sp}} = {\tau_0}_*\left( \operatorname{Td}(T\tau_0) {\tau_0'}^*(-) \right),$$  
and $$\UU^U = {\widehat \tau}_*\left( \operatorname{Td}(T\widehat  \tau) ({\widehat \tau'})^*(-) \right).$$

Note that for $t \in \AF^1$, $i_t^*: H^*_\ST(U) \to H^*_\ST(U_t)$ is an isomorphism.  The isomorphism $H^*_\ST(T_{sp}) \to H^*_{\ST}(T)$ may be realized as
 $i_1^* \circ (i_0^*)^{-1}$.  Similarly for $T'$ and $T_{sp}'$.
 
 Given $\widehat \alpha' \in H^*_\ST(U')$, 
 \begin{align*}
 \UU^T {i_1'}^* \widehat \alpha ' & = {\tau_1}_*\left( \operatorname{Td}(T\tau_1) {\tau_1'}^*  {i_1'}^* \widehat \alpha ' \right)   \\
 &=  {\tau_1}_*\left( \operatorname{Td}(T\tau_1)  {\widehat{i_1}}^* ({\widehat \tau '})^*   \widehat \alpha ' \right) \\
 &=  {\tau_1}_*\left(\widehat{ i_1}^* \operatorname{Td}(T\widehat \tau)   ({\widehat \tau '})^*   \widehat \alpha ' \right)\\
  &=  i_1^* {\widehat \tau}_*\left(\operatorname{Td}(T\widehat \tau)   ({\widehat \tau '})^*   \widehat \alpha ' \right) \\
  &= i_1^* \UU^U \widehat \alpha ',
 \end{align*}
 where the fourth equality is  naturality of the Gysin pushforward \cite[Appendix A.6 (3) (Naturality)]{AndFul}.  Similarly, over $t= 0$ we have $$\UU^{T_{sp}} {i_0'}^* \widehat \alpha ' =  i_0^* \UU^U \widehat \alpha '.$$   
 We may combine the above equations to conclude
 $$\UU^T {i_1 '}^* \circ ({i_0 '}^*)^{-1} \alpha ' =  {i_1 }^* \circ ({i_0 }^*)^{-1}\UU^{T_{sp}} \alpha '$$
 for  $\alpha ' \in H^*_\ST (T_{sp}')$.  \end{proof}

  Given an exact sequence $$0 \to B \to V' \to P \to 0,$$ an analogous statement holds with the same proof comparing $V'$ to $B \oplus P$.

Next, given $V, V' \to Z$,
let $\phi: Y  \to Z$ be the flag bundle map from Corollary~\ref{c:qsp}.
 Let $\ST = (\CC^*)^2$ act on $V \oplus V'$ as in the previous section, with the 1st factor of $\ST$ scaling $V$ and the 2nd factor scaling $V'$.
 By pulling back via $$\tilde \phi: T_Y = Y \times_Z T \to T$$ and applying Corollary~\ref{c:qsp}, we may assume without loss of generality that there exist sequences of inclusions of vector bundles 
\begin{align}\label{flag1}
     0 &\to A_1 \to A_2 \to \cdots \to A_r = V\\ \label{flag2}
     0 &\to B_1 \to B_2 \to \cdots \to B_s = V'\end{align}
such that the rank of successive quotients is 1.    
\begin{lemma}\label{l:pb}
We have
$$\Phi^{T_Y}_\ST \tilde \phi^* = \tilde \phi^* \Phi^T_\ST.$$
In particular, if $\tilde \phi^*\alpha$ is a strong asymptotic class then so is $\alpha$.

Furthermore, $\tilde \phi^* \UU^T = \UU^{T_Y} \tilde {\phi '}^*$.
\end{lemma}
\begin{proof}
The first statement is via the same argument as in Theorem~\ref{t:homom}.  For the second, we use the projection formula to note that for $\alpha \in H^*_\ST(T)$ and $\gamma \in H^*_\ST(T_Y)$, 
$$\br{ \gamma, \tilde \phi^* \Phi^T_\ST \alpha}^{T_Y} =  \br{\tilde \phi_* \gamma, \Phi^{T}_\ST \alpha}^{T}.$$  The second claim then follows because $\tilde \phi_*$ is surjective.

Consider the fiber diagram
$$\begin{tikzcd}
\widehat{T_Y} \ar[d, "\widehat{\tau}"] \ar[r, "\widehat{\phi}"] & \widehat{T} \ar[d, "\tau"] \\
T_Y \ar[r, "\tilde{\phi}"] & T.
\end{tikzcd}$$
Note that $\op{Td}(T \widehat \tau) = \widehat{\phi}^* \op{Td}(T \tau)$.
The final claim follows from this together with the fact $\tilde \phi^* \tau_* = \widehat \tau_* \widehat \phi^*$.
\end{proof}

Let $\bar V = \oplus_{i=1}^r A_i/A_{i-1}$ and $\bar V' = \oplus_{j=1}^s B_j/B_{j-1}$.  Let $\bar T = \tot( \bar V' \otimes \cc O_{\PP(\bar V)}(-1))$ and similarly for $\bar T'$.  
By applying Lemma~\ref{l:ctosplit} and its analogue for $V'$ repeatedly, we reach the following conclusion.
\begin{lemma}\label{l:r3}
      There are canonical isomorphisms
$$H^*_\ST(T_Y) \cong H^*_\ST(\bar T), \hspace{.5 cm} H^*_\ST(T'_Y) = H^*_\ST(\bar T').$$  
  Under these isomorphisms, we have the identifications
  $$\Phi^{T_Y}_\ST(q, z) = \Phi^{\bar T}_\ST(q, z)$$
  and $$\UU^{T_Y} = \UU^{\bar T}.$$
   \end{lemma}

  Lemmas~\ref{l:pb} and~\ref{l:r3} together allow us to generalize Corollary~\ref{p:stac2} to \emph{any} local model $T \to Z$, even when $V$ and $V'$ do not split.
    \begin{corollary}\label{c:stac3}
 The statements of corollary~\ref{p:stac2} holds for any local model $T \to Z$.
  \end{corollary}

The operator $S^T_\ST(q, z)$ and (therefore also $\Phi^T_\ST(q, z)$) has a well-defined non-equivariant limit, due to the fact that the evaluation map
 $\rm ev_i:\overline{ \mathcal M}_{0,n}(T, d) \to T$ is proper for $d$ in the span of $[L]$.  Let $\lim_{\underline \lambda \to 0}$ denote the nonequivariant limit.
Let
$\Phi^T_{+}(q, z)$ denote the power series obtained by subtracting the $d = 0$ summand: 
$$\Phi^T_{\ST, +}(q,  z) := \Phi^T_{\ST}(Q,\bt,z)\vline_{\substack{\;Q^d = 1 \text{ if } d = k[L]\\ \;Q^d = 0 \text{ otherwise}\\ \;\bt = \ln(q) c_1(T) }}
  - \Phi^T_{\ST}(Q,\bt,z)\vline_{\substack{\;Q^d = 0 \text{ for all } d\\ \;\bt =  \ln(q)c_1(T)}} .$$ 

\begin{lemma}\label{l:stac2}
For  $\alpha\in H^*_\ST(T)$,   if 
  $$ \lim_{\underline \lambda \to 0} |\!|
 \Phi^T_{\ST}(q, z)(\alpha)
|\!|
= O(|z|^{-m})$$
then 
$$ \lim_{\underline \lambda \to 0} |\!|
 \Phi^T_{\ST, +}(q, z)(\alpha)
|\!|
= O(|z|^{-m})$$
as $z \to 0$ along a ray $\mathbb{R}_+\cdot e^{\pi\sqrt{-1}\theta}$.
\end{lemma}
\begin{proof}
    We must show that  to show that for $\alpha \in H^*(T)$, $$|\!|\Phi^T(\underline Q=\underline 0,\bt = \ln(q)c_1(T),z) (\alpha)|\!|= O(|z|^{-m}).$$
    But  
    $$\Phi^T(\underline Q=\underline 0,\bt = \ln(q)c_1(T),z) (\alpha)= 
    (2\pi)^{-\frac{\dim X}{2}} q^{-c_1(T)/z} z^{-\mu} z^\rho (\alpha),$$
    which is bounded by a positive power of $z$.
\end{proof}
\subsection{Cohomological reduction to the local model}
Let $T \dasharrow T'$ be the local model for $X \dasharrow X'$.
Let $\Phi^X(q, z)$ and $\Phi^T(q, z)$ denote the fundamental solution matrices of \eqref{flat-sections} for $X$ and $T$ respectively, after restricting to the extremal curve classes as in \eqref{modified-pi-extremal2}.  
We first consider the asymptotic classes with respect to nonzero eigenvalues.
\begin{theorem}\label{t:nze}
View $\alpha \in H^*(Z)$ as a class in equivariant cohomology via the inclusion $H^*(Z) \hookrightarrow H^*_\ST(Z) = H^*(Z) \otimes R_\ST$.
Given $\gamma \in H^*(X)$, choose an equivariant lift $\widehat{i^*\gamma}$ of $i^*\gamma.$  Then
\begin{equation}
\label{support-on-F}
\br{\gamma, \Phi^X(q,z)\widehat \Gamma_X \Psi_m^X(\alpha)}^X = \lim_{\underline \lambda \to 0} \br{\widehat{i^*(\gamma)}, \Phi_\ST^T(q,z)\widehat \Gamma_T\Psi_m^T(\alpha)}^T_\ST.
\end{equation}
In particular,
$\widehat{\Gamma}_X\Psi_m^X(\alpha)$
is a strong asymptotic class with respect to $\lambda_m(q)$
and $\arg(z/q)={-2m-s\over r-s}\pi.$
\end{theorem}
\begin{proof} 
Note first that the class $\Phi^T(q,z)\widehat \Gamma_T\Psi_m^T(\alpha)$ is supported on $F \subset T$, so the nonequivariant limit in the right-hand side of \eqref{support-on-F} is well-defined.

Both sides of \eqref{support-on-F} are series in $q$. We will consider each coefficient of $q^d$ separately.
The degree 0 coefficients agree by the projection formula and the fact that $i^*\widehat \Gamma_X = \widehat \Gamma_T.$
    For $d>0$, the equality follows from the same  argument as in Theorem~\ref{t:func}. 
    In particular, if
    $d>0$ then  both sides are  equal to the integral over $[\overline{\cc M}_{0,2}(\PP(V), d)]^{vir} \cap \left(2 \pi \sqrt{-1}\right)^s (2 \pi)^{-\frac{\dim X}{2}} e(R^1\pi_* f^* V'(-1))$ of 
    $$\frac{{\rm ev_1}^{\PP(V), *}  z^{- \mu} z^\rho \left({\rm Ch}(j^*(\widehat \Gamma_X) \mathcal{O}_F(m)){\rm Td}(N_{F|X})^{-1} \pi^*(\alpha)\right)}{-z-\psi_1} \cup {\rm ev_2}^{\PP(V), *}j^* \gamma.
    $$
    
    The final statement then follows from Corollary~\ref{c:stac3}.
\end{proof}

\begin{prop}\label{pullbackXtoT} 
For $\alpha, \gamma \in H^*(X),$ let $\widehat{i^* \alpha}, \widehat{i^* \gamma}$ denote equivariant lifts of $i^*\alpha, i^*\gamma$.
Then the difference
$$\br{\gamma, \Phi^X(q,z)\widehat \Gamma_X \alpha}^X - \lim_{\underline \lambda \mapsto 0} \br{\widehat{i^* \gamma}, \Phi^T_{\ST,+}(q,z)\widehat \Gamma_T\widehat{i^* \alpha}}^T$$
    is polynomial in $z$.  

In particular, if $\widehat{i^* \alpha}$ is a strong tame asymptotic class with $\arg(z/q)={1-s\over r-s}\pi,$
    then so is $\alpha.$
  \end{prop}
\begin{proof}
The class $\Phi^T_{\ST, +}(q,z)\widehat \Gamma_T\widehat{i^* \alpha}$ is supported on $F \subset T$ so the nonequivariant limit on the right-hand side is well-defined.  The degree $d>0$ summands in the two terms are equal by the same consideration as in the proof of Theorem~\ref{t:nze}.  The degree $0$ summand  of $\br{\gamma, \Phi^X(q,z)\widehat \Gamma_X \alpha}^X$ is bounded by a positive power of $z$ by the same argument as in Lemma~\ref{l:stac2}.

The final statement is then an immediate consequence of  Lemma~\ref{l:stac2}.\end{proof}

\subsection{Comparing the Fourier-Mukai transform to the local model}
In this section we compare how the  Fourier--Mukai transform (both $K$-theoretic and cohomological) interacts with the restriction to the local model.

Let 
$$\FM: \tau_* {\tau'}^*: K^0(X') \to K^0(X)$$ 
denote the Fourier--Mukai transform for $X$ and let 
$$\FM^{T} := {\tau_T}_*\tau_{T'}^*: K^0(T') \to K^0(T)$$ be the Fourier--Mukai transform between local models.  
Let $U$ (resp. $U'$, $\widehat U$) denote the deformation to the normal cone of $F$ in $X$ (resp. $F'$ in $X'$).  More precisely, 
$$U = \operatorname{Bl}_{F\times \{0\}}X \times \AF^1 \setminus \operatorname{Bl}_{F}X$$ and similarly for $U',$ $\widehat U$.  Recall that the fiber of $U$ over $1 \in \AF^1$ is $X$, and the fiber of $U$ over $0$ is $T.$  
The next lemma shows that the resolution $X \xleftarrow{} \widehat X \to X'$ of the flip $X \dasharrow X'$ is compatible with deformation to the normal cone.
\begin{lemma}
Define $\widehat  U$ to be the deformation to the normal cone of $F \times_Z F'$ in $\widehat X$.
    Then we have the following commutative diagram of deformations to the normal cone:
    \[
    \begin{tikzcd}
      &  & \ar[dl, swap, "\widehat \tau"] \ar[dr, "\widehat \tau'"] \widehat U &&\\
  F \times \AF^1 \ar[r]   &   U \ar[dr] && \ar[dl] U'&  F' \times \AF^1 \ar[l]\\
       && \AF^1 && 
    \end{tikzcd}
    \]
    Furthermore, $\widehat U$ is canonically isomorphic to $\operatorname{Bl}_{F \times \AF^1} U = \operatorname{Bl}_{F' \times \AF^1} U'.$
\end{lemma}
\begin{proof}
    By the universal property of blowups, there is a map 
$$\operatorname{Bl}_{F \times \AF^1}( \operatorname{Bl}_{F \times \{0\}} X \times \AF^1) \to 
    \operatorname{Bl}_{F \times \AF^1} X \times \AF^1 = \widehat X \times \AF^1.$$
    Under this map, the preimage of $F \times_Z F' \times \{0\}$ is a Cartier divisor, so there is an induced map
    $$\operatorname{Bl}_{F \times \AF^1}( \operatorname{Bl}_{F \times \{0\}} X \times \AF^1) \to 
    \operatorname{Bl}_{F \times_Z F' \times \{0\}}( \widehat X \times \AF^1).$$
By similar reasoning, there is also a map 
  $$\operatorname{Bl}_{F \times_Z F' \times \{0\}}( \widehat X \times \AF^1) \to \operatorname{Bl}_{F \times \AF^1}( \operatorname{Bl}_{F \times \{0\}} X \times \AF^1),$$
   and these are inverse to each other.  From this it is easy to identify $\widehat U$ with $\operatorname{Bl}_{F \times \AF^1} U.$  This defines $\widehat \tau$.  The argument for $\widehat \tau'$ is identical.
\end{proof}
Let $j_{0/1}: F \to U_{0/1}$ denote the inclusion of $F$ into the fiber of $U$ over $0/1$ (note that $U_1 = X$ and $U_0 = T$), and similarly for $j_{0/1}': F' \to U_{0/1}'.$  Let $i_{0/1}:  U_{0/1} \to U$ denote the inclusion of the respective fibers.  
    \[
    \begin{tikzcd}
        & F' \ar[drr, bend left, "j'_1=j'"]  \ar[d, "j'_0"] & &\\
      & T'=U'_0 \ar[r, "i'_0"] \ar[u, bend left,  "k'"] \ar[rr, bend left, "i'"]& U' & X'=U'_1 \ar[l, swap, "i'_1"]\\
      & \widehat{T}   \ar[r, "\widehat i_0"] \ar[d, swap, "\tau_0"] \ar[u, "\tau'_0"]& \widehat{U} \ar[d, swap, "\widehat \tau"] \ar[u, "\widehat {\tau'}"]& \widehat{X} \ar[l, swap, "\widehat i_1"] \ar[d, swap, "\tau_1"] \ar[u, "\tau'_1"]\\
      & T=U_0 \ar[r, swap, "i_0"] \ar[rr, bend right, swap, "i"] \ar[d, swap, bend right, "k"]& U  & X=U_1 \ar[l, "i_1"]\\
        & F \ar[urr, bend right, swap, "j_1=j"] \ar[u, swap, "j_0"]& &
    \end{tikzcd}
    \]

Define 
$$\UU:=\tau_*\left( \operatorname{Td}(T\tau) {\tau'}^*(--) \right): H^*(X') \to H^*(X).$$
Then by Grothendieck--Riemann--Roch we have 
$$\ch \circ\ \FM = \UU \circ \ch.$$ We let $\UU^T: H^*(T') \to H^*(T)$ denote the associated map for $T$.  We will show that the map $\UU$  ``commutes''  with restriction to the local model.
\begin{prop}\label{cohres}
For all $\alpha ' \in H^*(X')$,
\begin{equation}
\UU^T  {i '}^* (\alpha ') = i^*  \UU (\alpha').
\end{equation}
\end{prop}

\begin{proof}
The argument is similar to that used in Lemma~\ref{l:ctosplit}, but slightly more involved.  In this case we will use the deformation to the normal cone construction given above.

Given $\alpha \in H^*(X),$ let $\widehat \alpha$ denote the pullback to $U$.  Note that $\alpha = i_1^*(\widehat \alpha)$.  From this  and the fact that \begin{equation}\label{e:tr} j_1^* i_1^* = j_0^* i_0^*: H^*(U) \to H^*(F),\end{equation}
we see that 
\begin{align}\label{e:tr2}
i^* \alpha  = k^* j_1^* i_1^* \widehat \alpha 
= k^* j_0^* i_0^* \widehat \alpha 
 = i_0^* \widehat \alpha.
\end{align}
We have analogous statements for  $\alpha' \in H^*(X')$.

By \cite[Appendix A.6 (3) (Naturality)]{AndFul}, pushforward commutes with specialization, therefore, 
for $t=0,1$ and $\gamma \in H^*(\widehat U)$, 
\begin{align}\label{pushspec}
   i_t^* \widehat \tau_* \left(\operatorname{Td}(\widehat \tau) \gamma\right) 
    &= {\tau_t}_* \widehat i_t^* \left(\operatorname{Td}(\widehat \tau) \gamma\right) \\ \nonumber
    &= {\tau_t}_*\left(\operatorname{Td}(\tau_t)  \widehat i_t^* \gamma\right),
\end{align}
where $\tau_1 = \tau$ and $\tau_0 = \tau_T$.  Here we  also use the fact that $\widehat i_t^* \operatorname{Td}(\widehat \tau) = \operatorname{Td}\tau_t$ for $t = 0, 1$.  For $t=1$ this is obvious, as $\widehat i_1^* T \widehat U = TX \times \CC$ and $i_1^* TU = TX\times \CC$.  For $t= 0$, we use a similar observation as \eqref{e:tr2}:
$$ \widehat i_0^* T\widehat \tau = \widehat i^* \widehat i_1^* T\widehat \tau = \widehat i^* T\tau = T\tau_0.$$

By \eqref{e:tr} and \eqref{pushspec}, we see that 
$$ j_1^* {\tau_1}_* \widehat i_1^* \widehat \alpha =  j_0^* {\tau_0}_* \widehat i_0^* \widehat \alpha .$$
For $\alpha' \in H^*(X'),$ let $\widehat \alpha'$ denote the pullback to $U'$.
\begin{align*}
     i^* \UU \alpha' & =  k^* j_1^* {\tau_1}_*\left(\operatorname{Td}(\tau) {\tau_1'}^* \alpha'\right) \\
    & =  k^* j_1^* {\tau_1}_* \widehat i_1^* \left(\operatorname{Td}(\widehat \tau) {\widehat {\tau'}}^* \widehat \alpha'\right) \\ & =  k^* j_0^* {\tau_0}_* \widehat i_0^* \left(\operatorname{Td}(\widehat \tau) {\widehat {\tau'}}^* \widehat \alpha' \right)\\
    & =  k^* j_0^* {\tau_0}_*  \left( \operatorname{Td}(\tau_0)  {\tau'_0}^* {i'_0}^* \widehat \alpha' \right)\\
    &=   {\tau_0}_* \left(\operatorname{Td}(\tau_0){\tau_0'}^* {i'}^*  \alpha'  \right)\\
    &=  \UU^T {i'}^*  \alpha' .
\end{align*}
\end{proof}

\begin{corollary}\label{c:rFM}
Restriction to the local model ``commutes'' with the Fourier--Mukai transform after applying the modified Chern character.  That is, we have
\begin{equation} \label{dfm}
   \ch i^* \FM = \ch \FM^{T} {i'}^*.
\end{equation} 
\end{corollary}

This allows us to characterize the strong tame asymptotic classes of $X$.
\begin{theorem}\label{t:tac}
    For any $\alpha' \in H^*(X')$, the cohomology class $\widehat \Gamma_X \UU(\alpha')\in H^*(X)$ is a strong tame asymptotic class with $\arg(z/q)={(1-s)\over r-s}\pi.$
\end{theorem}
\begin{proof}
By Proposition~\ref{pullbackXtoT}, it suffices to consider an equivariant lift of $\widehat{i^* \UU(\alpha')}$ and show that $\widehat \Gamma_T\widehat{i^* \UU(\alpha')}$ is a strong asymptotic class.  By Proposition~\ref{cohres}, this is equal to $  \widehat \Gamma_T\UU^T \widehat{{i'}^*(\alpha ')}$, which  is a strong tame asymptotic class along the given ray by Corollary~\ref{c:stac3}.
\end{proof}

The two statements of Theorem~\ref{main-theorem} are given by
Theorem~\ref{t:nze} and Theorem~\ref{t:tac}, respectively.

\appendix

\section{Asymptotic expansions of Meijer G-functions}
\label{appendix-asymptotic}

Following ~\cite[Section 1]{Meijer}\cite[Section 5.2]{Luke}, 
for any $(a_1, \ldots, a_p, b_1, \ldots, b_q)\in \mathbb{C}^{p+q}$ 
such that 
$$\{a_k-b_j\mid 1\leq k\leq p, 1\leq j\leq q\}\cap \mathbb{Z}_+=\emptyset,$$ 
we consider the Meijer $G$-function
\begin{equation}
\label{meijer-def}
G_{p,q}^{m,n}
\left(
\begin{array}{l}
a_1, \ldots, a_p\\
b_1, \ldots, b_q
\end{array}
\bigg\vert\ t
\right)
={1\over 2\pi\sqrt{-1}}
\int_L{\prod\limits_{j=1}^{m}\Gamma(b_j-x)\prod\limits_{j=1}^{n}\Gamma(1-a_j+x)\over\prod\limits_{j=m+1}^{q}\Gamma(1-b_j+x)\prod\limits_{j=n+1}^{p}\Gamma(a_j-x)}t^x dx,
\end{equation}
where the path $L$ is chosen appropriately as in \cite[Section 5.2]{Luke}. 
In particular, if $m+n-{1\over 2}(p+q)>0$, the path $L$ goes from $-\sqrt{-1}\infty$ to $\sqrt{-1}\infty$ so that 
all poles of $\Gamma(b_j-x)$, $j=1, 2, \ldots, m$, lie to the right of the path, 
and all poles of $\Gamma(1-a_k+x)$, $k=1,2,\ldots, n,$ lie to the left of the path.

Let $\nu:=q-p$. Let $\epsilon=1$ if $\nu>1$ and $\epsilon=1/2$ if $\nu=1$.
The following results of asymptotic expansions are proved by Barnes in \cite{Barnes} and reformulated by Meijer in \cite{Meijer} via Meijer $G$-functions. 

\begin{theorem}\label{exponential-asymptotic}
\cite[Theorem C]{Meijer} 
\cite[Section 5.7 Theorem 5]{Luke}
There exists an asymptotic expansion 
\begin{equation}
\label{asymptotic-solution-meijer}
G_{p,q}^{q,0}
\left(
\begin{array}{l}
a_1, \ldots, a_p\\
 b_1,\ldots, b_q
\end{array}  \bigg\vert\ t\right)
\sim {(2\pi)^{\nu-1\over 2}\over \nu^{1\over 2}}
 \exp\left(-\nu\ t^{1\over \nu}\right)
 t^{\theta}\left(1+\sum_{k=1}^{\infty} M_k t^{-k/\nu}\right)
\end{equation}
as $t\to \infty$ with $|\arg t|< (\nu+\epsilon)\pi$.
Here $M_k$ are constants determined recursively and 
 $$\theta={1\over \nu} \left({1-\nu\over 2}+\sum_{i=1}^{q}b_i-\sum_{j=1}^{p}a_j\right).$$
\end{theorem}

\begin{theorem}
\label{algebraic-asymptotic}
    \cite[Theorem A]{Meijer}
    \cite[Section 5.7 Theorem 1]{Luke}
    The Meijer $G$ function $$
G_{p,q}^{q,1}\left(t || a_l\right):=G_{p,q}^{q,1}\left(
\begin{array}{c}
a_l, a_1, \ldots, a_{l-1}, a_{l+1}, \cdots, a_p\\
b_1, \ldots,b_q
\end{array}
\bigg\vert\ t\right)
$$
admits for large values of $|t|$ with $|\arg(t)|<{\nu\over 2}\pi+\pi$ the asymptotic expansion 
$$G_{p,q}^{q,1}\left(t || a_l\right)\sim E_{p,q}(t || a_l):=t^{-1+a_l}\sum_{d=0}^{\infty} {(-t)^{-d}\prod\limits_{i=1}^{q}\Gamma(1+b_i-a_l+d)\over \prod\limits_{j=1}^{p}\Gamma(1+a_j-a_l+d)}.$$
\end{theorem}


\vskip .1in

\noindent{\small Department of Mathematics, University of Oregon, Eugene, OR 97403,
USA}

\noindent{\small E-mail: yfshen@uoregon.edu}

\vskip .1in

\noindent{\small Department of Mathematics, Colorado State University, Fort Collins, CO 80523,
USA}

\noindent{\small E-mail: mark.shoemaker@colostate.edu}

\end{document}